\documentclass[12pt]{amsart}

\usepackage{amssymb}
\usepackage{bm}
\usepackage{graphicx}
\usepackage{psfrag}
\usepackage[usenames,dvipsnames]{xcolor}
\usepackage{enumerate}
\usepackage{url}

\usepackage{algpseudocode}

\usepackage{mathtools}

\usepackage{xy}
\input xy
\xyoption{all}

\usepackage{tikz}
\usetikzlibrary{positioning} 
\usetikzlibrary{graphs,graphs.standard}

\newcommand{\excise}[1]{}

\newtheorem{thm}{Theorem}[section]
\newtheorem{lemma}[thm]{Lemma}

\newtheorem{cor}[thm]{Corollary}
\newtheorem{prop}[thm]{Proposition}
\newtheorem{conj}[thm]{Conjecture}

\theoremstyle{definition}

\newtheorem{example}[thm]{Example}
\newtheorem{remark}[thm]{Remark}
\newtheorem{defn}[thm]{Definition}

\newtheorem{notation}[thm]{Notation}

\numberwithin{equation}{section}

\newcommand{\ring}[1]{\ensuremath{\mathbb{#1}}}

\renewcommand\>{\rangle}
\newcommand\<{\langle}

\newcommand\CC{\ring{C}}

\newcommand\QQ{\ring{Q}}
\newcommand\RR{\ring{R}}

\newcommand\ZZ{\ring{Z}}

\DeclareMathOperator\Betti{Betti} 
\DeclareMathOperator\Cong{Cong} 

\DeclareMathOperator\Ap{Ap} 

\begin{document}

\mbox{}
\title{On parametrized families of numerical semigroups}

\author[F.~Kerstetter]{Franklin Kerstetter}
\address{Mathematics Department\\University of California Davis\\Davis, CA 95616}
\email{fjkerstetter@ucdavis.edu}

\author[C.~O'Neill]{Christopher O'Neill}
\address{Mathematics Department\\San Diego State University\\San Diego, CA 92182}
\email{cdoneill@sdsu.edu}


\date{\today}

\begin{abstract}
A numerical semigroup is an additive subsemigroup of the non-negative integers.  In this paper, we consider parametrized families of numerical semigroups of the form $P_n = \<f_1(n), \ldots, f_k(n)\>$ for polynomial functions $f_i$.  We conjecture that for large $n$, the Betti numbers, Frobenius number, genus, and type of $P_n$ each coincide with a quasipolynomial.  This conjecture has already been proven in general for Frobenius numbers, and for the remaining quantities in the special case when $P_n = \<n, n + r_2, \ldots, n + r_k\>$.  Our main result is to prove our conjecture in the case where each $f_i$ is linear.  In the process, we develop the notion of weighted factorization length, and generalize several known results for standard factorization lengths and delta sets to this weighted setting.  
\end{abstract}

\maketitle


\section{Introduction}
\label{sec:intro}

A numerical semigroup $S$ is an additively closed subset of $\ZZ_{\ge 0}$, usually specified using a generating set $r_1, \ldots, r_k$, i.e.,
$$S = \<r_1, \ldots, r_k\> = \{z_1r_1 + z_2r_2 + \cdots + z_kr_k \mid z_1, \ldots, z_k \in \ZZ_{\ge 0}\}.$$
Many classical problems surrounding numerical semigroups involve arithmetic invariants, such as the Frobenius number $\mathsf F(S)$, genus $\mathsf g(S)$, type $\mathsf t(S)$, and delta set $\Delta(S)$, each of which is difficult to compute when the generators of $S$ are large.  
For a thorough introduction to numerical semigroups, see~\cite{numerical}.  

This paper considers parametrized families of numerical semigroups of the form
$$P_n = \<f_1(n), \ldots, f_k(n)\>$$
for some functions $f_1(n), \ldots, f_k(n)$.  Such families have arisen in two main settings in the last decade.  First is the \emph{parametric Frobenius problem}, which asks under what conditions the function $n \mapsto \mathsf F(P_n)$ coincides with a quasipolynomial (that is, a polynomial with periodic coefficients) for large $n$.  It was conjectured in~\cite{rouneparametricfrob} that this holds whenever the functions $f_i$ are themselves polynomials, where this was proven in the case where $\deg f_i = 1$ for all $i$, as well as in the case $k = 3$.  This appears to have been proven in general~\cite{shenparametricfrob}, though the results have yet to appear outside the \texttt{arXiv}, and the authors of this manuscript have been unable to contact the author.  

Separately, \emph{shifted} numerical semigroups, which have a specialized parametrization 
$$M_n = \<n, n + r_2, \ldots, n + r_k\>$$
for positive integers $r_2, \ldots, r_k$, have been examined in numerous recent papers.  It~is known that the delta set of $M_n$ is eventually periodic~\cite{shiftydelta}, and that the Frobenius number, genus, and type of $M_n$ are each eventually quasipolynomial~\cite{shiftedaperysets}.  Additionally, the minimal relations between the generators of $M_n$, usually studied in the form of minimal presentations \cite{numerical} or syzygies of the defining toric ideal \cite{cca}, are known to satisfy a certain periodicity originally conjectured by Herzog and Srinivasan and proven by Vu~\cite{vu14}.  These results were later improved in~\cite{shiftyminpres}, wherein several consequences for other semigroup invariants were also derived, and further specialized in~\cite{shiftedtangentcone,shifted3gen}.  

The results mentioned above provide ample evidence of a more general phenomenon, which we now conjecture formally.  

\begin{conj}\label{conj:main}
If $f_1, \ldots, f_k:\ZZ \to \ZZ$ are eventually increasing polynomials and
$$P_n = \<f_1(n), \ldots, f_k(n)\>,$$
then $\Betti(P_n)$ is eventually quasipolynomial in $n$.  As a consequence, the Frobenius number, genus, and type of $P_n$ are each eventually quasipolynomial in $n$.  
\end{conj}

Note that the word ``consequence'' in Conjecture~\ref{conj:main} is intended as an informal~claim.  In~particular, the main results of~\cite{shiftyminpres,shiftedaperysets} for shifted numerical semigroups (where the conjecture is already proven) stem from a single underlying result (\cite[Theorem~3.4]{shiftyminpres}) regarding the Betti elements of $P_n$ (that is, elements whose factorizations encode the minimal relations between the generators of $P_n$).  Conjecture~\ref{conj:main} claims this core behavior occurs more generally, and that the remaining claims follow as consequences.  

\begin{remark}\label{r:conjproof}
After posting this manuscript, a proof of the ``eventually quasipolynomial'' claims in Conjecture~\ref{conj:main} appeared elsewhere on the arXiv~\cite{parametricpresburgerbigthm}.  The results therein are broad, with most claims extended to parametrized families of affine semigroups, but the proofs are nonconstructive, relying on formal logic and Presburger arithmetic.  As~such, the informal ``consequence'' claim discussed above remains open.  
\end{remark}

In this paper, we prove Conjecture~\ref{conj:main} in the case where the functions $f_1(n), \ldots, f_k(n)$ are linear.  The main results are in Sections~\ref{sec:minpres} and~\ref{sec:aperysets}, which generalize results for shifted numerical semigroups that appeared in~\cite{shiftyminpres} and~\cite{shiftedaperysets}, respectively.  The results in those sections follow from a central result about Betti elements for large $n$ (Theorem~\ref{t:mesalemma}), providing the ``consequently'' part of Conjecture~\ref{conj:main}.  As a necessary step in stating our main results, we develop the notion of ``weighted factorization length'' in Section~\ref{sec:weightedlengths}, and generalize several known results involving standard factorization length.  As evidence of the generality in Conjecture~\ref{conj:main}, we close this paper with Example~\ref{e:nonlinear}, a non-linear example where Conjecture~\ref{conj:main} appears to hold.  

\subsection*{Acknowledgements}
The authors would like to thank Scott Chapman and Pedro Garc\'ia-S\'anchez for their helpful comments and suggestions.

\section{Numerical semigroups and factorization length}
\label{sec:background}

In this section, we state some background definitions for factorizations of numerical semigroup elements; the books~\cite{nonuniq} and~\cite{numerical} contain thorough introductions to nonunique factorization and numerical semigroups, respectively.  Several of the quantities in Definition~\ref{d:numerical} involving (unweighted) factorization length have a weighted generalization introduced in subsequent sections of this paper.

\begin{defn}\label{d:numerical}
A \emph{numerical semigroup} $S$ is an additive subsemigroup of $\ZZ_{\ge 0}$ (note, we do \textbf{not} require $S$ to have finite complement).  We write 
$$S = \<r_1, \ldots, r_k\> = \{z_1r_1 + \cdots + z_kr_k : z_1, \ldots, z_k \in \ZZ_{\ge 0}\}$$
for the semigroup generated by $r_1, \ldots, r_k$.  
A \emph{factorization} of $n \in S$ is an expression 
$$n = z_1r_1 + \cdots + z_kr_k$$
of $n$ as a sum of generators of $S$, and the \emph{length} of a factorization is the sum $z_1 + \cdots + z_k$.  The \emph{set of factorizations} of $n$ is the set
$$\mathsf Z_S(n) = \{z \in \ZZ_{\ge 0}^k : n = z_1r_1 + \cdots + z_kr_k\}$$
viewed as a subset of $\ZZ_{\ge 0}^k$, and the \emph{length set} of $n$ is the set 
$$\mathsf L_S(n) = \{z_1 + \cdots + z_k : z \in \mathsf Z_S(n)\},$$
of all possible factorization lengths of $n$.  Writing $\mathsf L_S(n) = \{\ell_1 < \cdots < \ell_m\}$, define 
$$\Delta_S(n) = \{\ell_i - \ell_{i-1} : 2 \le i \le m\} \qquad \text{and} \qquad \Delta(S) = \bigcup_{n \in S} \Delta_S(n)$$
as the \emph{delta sets} of $n$ and $S$, respectively.  
The \emph{maximum} and \emph{minimum} factorization length functions are defined as 
$$\mathsf M_S(n) = \max \mathsf L_S(n) \qquad \text{ and } \qquad \mathsf m_S(n) = \min \mathsf L_S(n),$$
respectively.  
\end{defn}

We state two results from the literature (Theorems~\ref{t:deltaset} and~\ref{t:maxminorig}) that we will generalize in the next section.  The first result depends on the following definition.  

\begin{defn}\label{d:bettielement}
Given a numerical semigroup $S$ and an element $n \in S$, the \textit{factorization graph} of $n$, denoted $\nabla_n$, has vertex set $\mathsf Z(n)$, and two vertices $z,z' \in \mathsf Z(n)$ are connected by an edge whenever they have at least one generator in common.  We say $n$ is a \emph{Betti element} of $S$ if $\nabla_n$ is disconnected.  Define
$$\Betti(S) = \{n \in S \mid n \text{ is a Betti element of } S\}.$$
\end{defn}

\begin{example}\label{e:bettielements}
The Betti elements of $S = \<6, 9, 20\>$ are $\Betti(S) = \{18,60\}$, whose factorization graphs are depicted in Figure~\ref{f:bettielements}.  As we will see in Section~\ref{sec:minpres}, these elements encode the minimal relations between the generators of $S$:\ $18$ is the smallest element that can be factored using $6$ and $9$, and $60$ is the smallest element that can be factored using $6$ and $9$ and separately using $20$.  
\end{example}

\begin{figure}
\begin{center}
\includegraphics[height=1.2in]{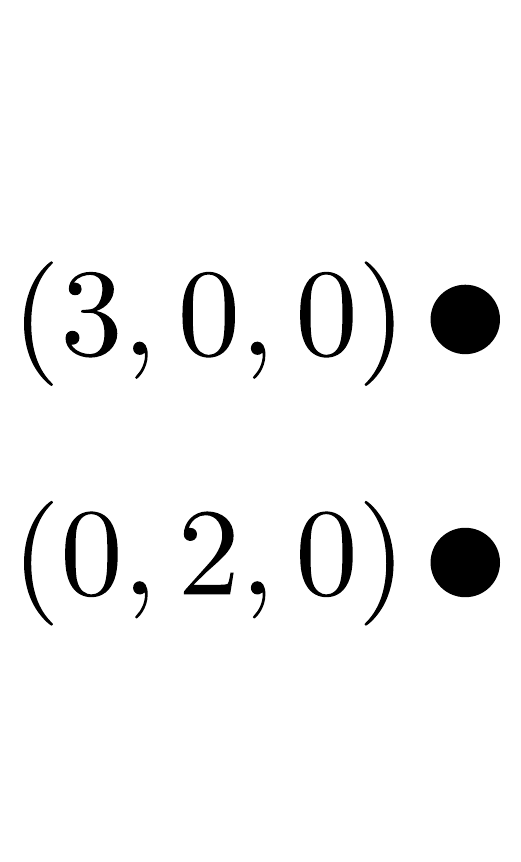}
\hspace{1.0in}
\includegraphics[height=1.2in]{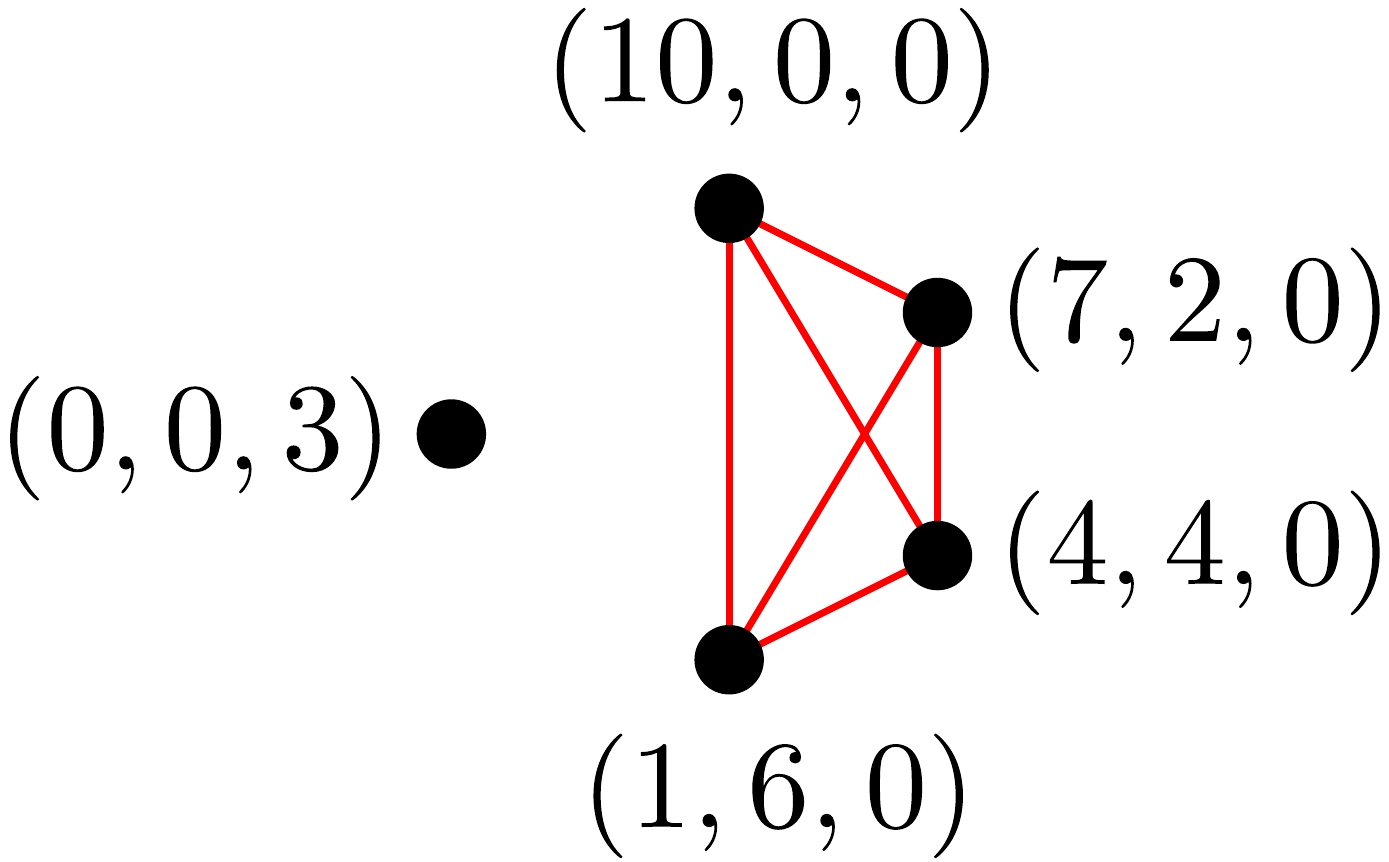}
\end{center}
\caption{The factorization graphs $\nabla_{18}$ (left) and $\nabla_{60}$ (right) in the numerical semigroup $S = \<6, 9, 20\>$ from Example~\ref{e:bettielements}.}
\label{f:bettielements}
\end{figure}

\begin{remark}\label{r:bettinumbers}
From a commutative algebra viewpoint, Betti elements coincide with graded degrees of the minimal generators of toric ideals.  Given a numerical semigroup $S = \<r_1, \ldots, r_k\>$, the kernel $I = \ker\varphi$ of the ring homomorphism determined by
$$\begin{array}{r@{}c@{}l}
\varphi:\CC[x_1, \ldots, x_k] &{}\to{}& \CC[y] \\
x_i &{}\mapsto{}& y^{r_i}
\end{array}$$
is the defining toric ideal of $S$.  As an example, if $S = \<6, 9, 20\>$, then the defining toric ideal $I \subset \CC[x,y,z]$ has precisely 4 minimal generating sets, namely
$$\{x^3 - y^2, x^{10} - z^3\}, \, \{x^3 - y^2, x^7y^2 - z^3\}, \,
\{x^3 - y^2, x^4y^4 - z^3\}, \, \text{and} \, \{x^3 - y^2, xy^6 - z^3\},$$
each of which has one homogeneous element of degree $18$ and one of degree $60$ (here, the graded degree of each monomial is determined by its image under $\varphi$).  This matches the Betti elements $\Betti(S) = \{18, 60\}$ obtained in Example~\ref{e:bettielements}.  
\end{remark}

\begin{thm}[{\cite{deltasets,bfdelta}}]\label{t:deltaset}
For any numerical semigroup $S = \<r_1, \ldots, r_k\>$, the set $\Delta(S)$ is~nonempty and finite, and $\gcd \Delta(S) = \min \Delta(S)$.  Moreover,
$$\min \Delta(S) = \gcd\{r_i - r_{i-1} : 2 \le i \le k\}$$
and
$$\max \Delta(S) = \max_{n \in \Betti(S)} \max \Delta_S(n).$$
\end{thm}

\begin{thm}[{\cite{elastsets,shiftyminpres}}]\label{t:maxminorig}
For $n > r_k^2$ in a numerical semigroup $S = \<r_1, \ldots, r_k\>$, we have
$$\mathsf M(n + r_1) = \mathsf M(n) + 1 \qquad \text{and} \qquad \mathsf m(n + r_k) = \mathsf m(n) + 1.$$
\end{thm}

The functions in Theorem~\ref{t:maxminorig} are said to coincide for large $n$ with \emph{quasipolynomials}, that is, polynomial functions $\ZZ \to \RR$ with periodic coefficients.  In particular, 
$$\mathsf M(n) = \tfrac{1}{r_1}n + a(n) \qquad \text{and} \qquad \mathsf m(n) = \tfrac{1}{r_k}n + b(n)$$
for some periodic functions $a(n)$ and $b(n)$ with periods $r_1$ and $r_k$, respectively.

\section{Weighted factorization lengths}
\label{sec:weightedlengths}

Before examining parametrized families of numerical semigroups, we introduce a generalization of factorization length that independently weights each generator and plays a key role in the results of subsequent sections.  We give two main results in this section, each of which generalizes existing results for the usual factorization length.  The~first is Theorem~\ref{t:maxminquasi}, which generalizes \cite[Theorems~4.2 and~4.3]{elastsets} and joins a growing family of ``eventually quasipolynomial'' results concerning factorization length (see \cite{factorhilbert} and the references therein for an overview).  The second is Theorem~\ref{t:weighteddelta}, which gives weighted versions of \cite[Lemma~3]{arithmeticintdom} and \cite[Theorem 2.5]{bfdelta}, both of which are central to the study of delta sets.  

\begin{defn}\label{d:weightedlength}
Fix a numerical semigroup $S = \<r_1, \ldots, r_k\>$ and a rational vector $w = (w_1, \ldots, w_k) \in \QQ^k$ of \emph{weights}.  
Given $n \in S$ and $z = (z_1, \ldots, z_k) \in \mathsf Z(n)$, the \textit{weighted length} of $z$ is
$$|z|_w = w \cdot z = w_1z_1 + \cdots + w_kz_k,$$
and the \emph{weighted length set} of $n$ is
$$\mathsf L_{S,w}(n) = \{|z|_w : z \in \mathsf Z(n)\}.$$
The maps $\mathsf M_w: S \mapsto \QQ$ and $\mathsf m_w: S \mapsto \QQ$ given by
$$\mathsf M_w(n) = \max \mathsf L_{S,w}(n) \qquad \text{and} \qquad \mathsf m_w(n) = \min \mathsf L_{S,w}(n)$$
are the \emph{maximum weighted length} and \emph{minimum weighted length} functions, respectively.  
\end{defn}

\begin{defn}\label{d:worder}
Fix a numerical semigroup $S = \<r_1, \ldots, r_k\>$ and a weight vector $w =~(w_1, \ldots, w_k) \in \QQ^k$.  
The \emph{$w$-ordering} $\le_w$ on $\{r_1, \ldots, r_k\}$ is defined so that
$$r_i \le_w r_j \qquad \text{whenever} \qquad w_i/r_i \ge w_j/r_j.$$
Note that the $w$-ordering is transitive, but need not be a total (or even partial) ordering, as $r_i =_w r_j$ is possible for $r_i \ne r_j$.  
\end{defn}

\begin{remark}\label{r:standardlengthweighted}
The standard length $|\!\cdot\!|$ can be viewed as a special case of weighted length $|\cdot|_w$ with weight vector $w = (1,\ldots, 1)$.  In this case, the $w$-ordering on $r_1, \ldots, r_k$ is the usual total ordering in $\ZZ$.  
\end{remark}

\begin{example}\label{e:weightedlength}
Let $S = \<6, 9, 20\>$.  For the weight vector $w = (3,1,4)$, the $w$-ordering on the generators of $S$ is $6 <_w 20 <_w 9$ since $\tfrac{3}{6} > \tfrac{4}{20} > \tfrac{1}{9}$.  The same $w$-ordering is induced by $w = (3,-1,4)$, but some factorizations have negative weighted length, e.g.~$(2,12,1) \in \mathsf Z_S(140)$ has $|(2,12,1)|_w = -2$.  Figure~\ref{f:weightedlength} depicts $\mathsf m_{S,w}(-)$ for both weight vectors; evident is the eventually quasilinear property implied by Theorem~\ref{t:maxminquasi}.  
\end{example}

\begin{figure}
\begin{center}
\includegraphics[width=2.8in]{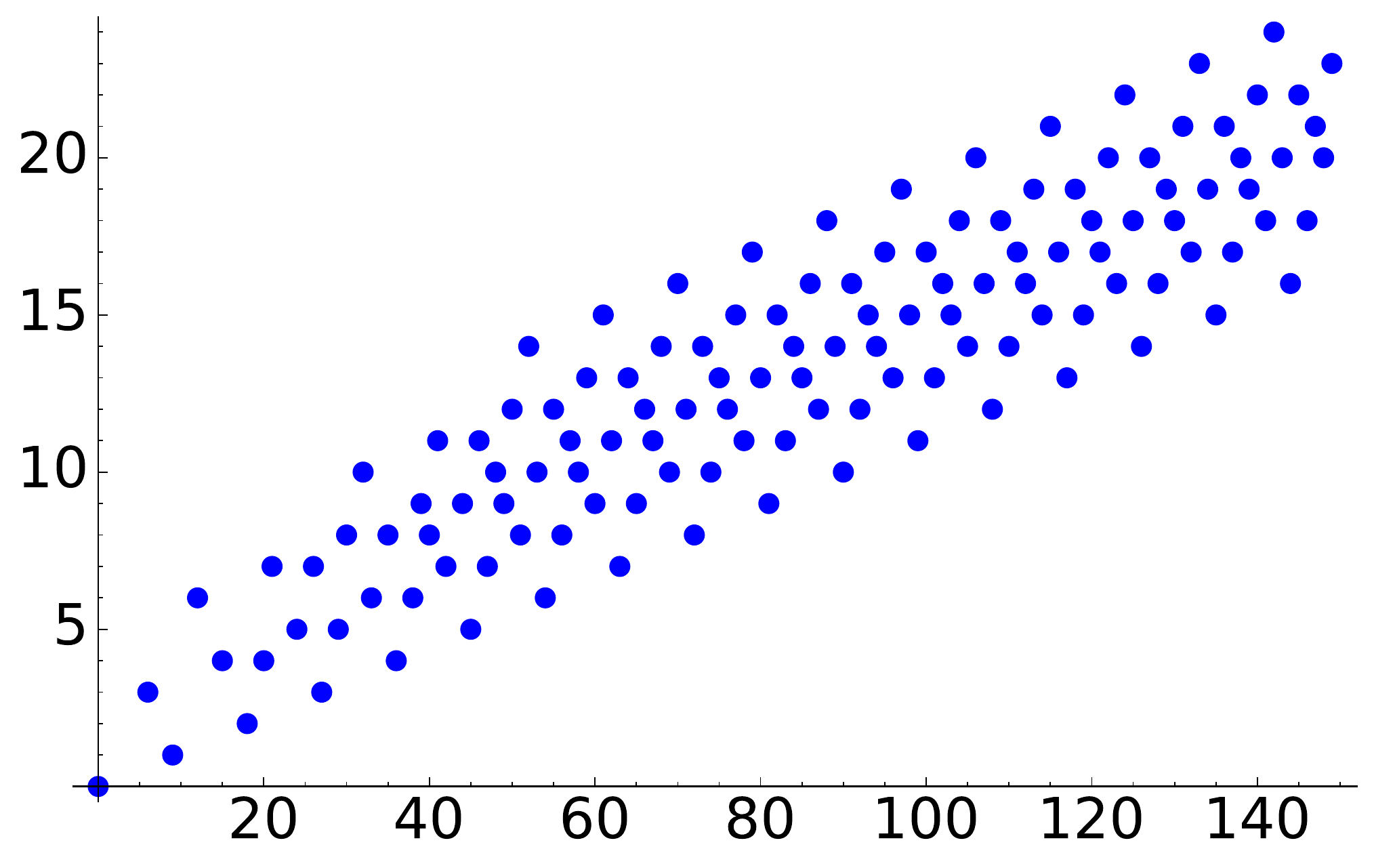}
\hspace{0.2in}
\includegraphics[width=2.8in]{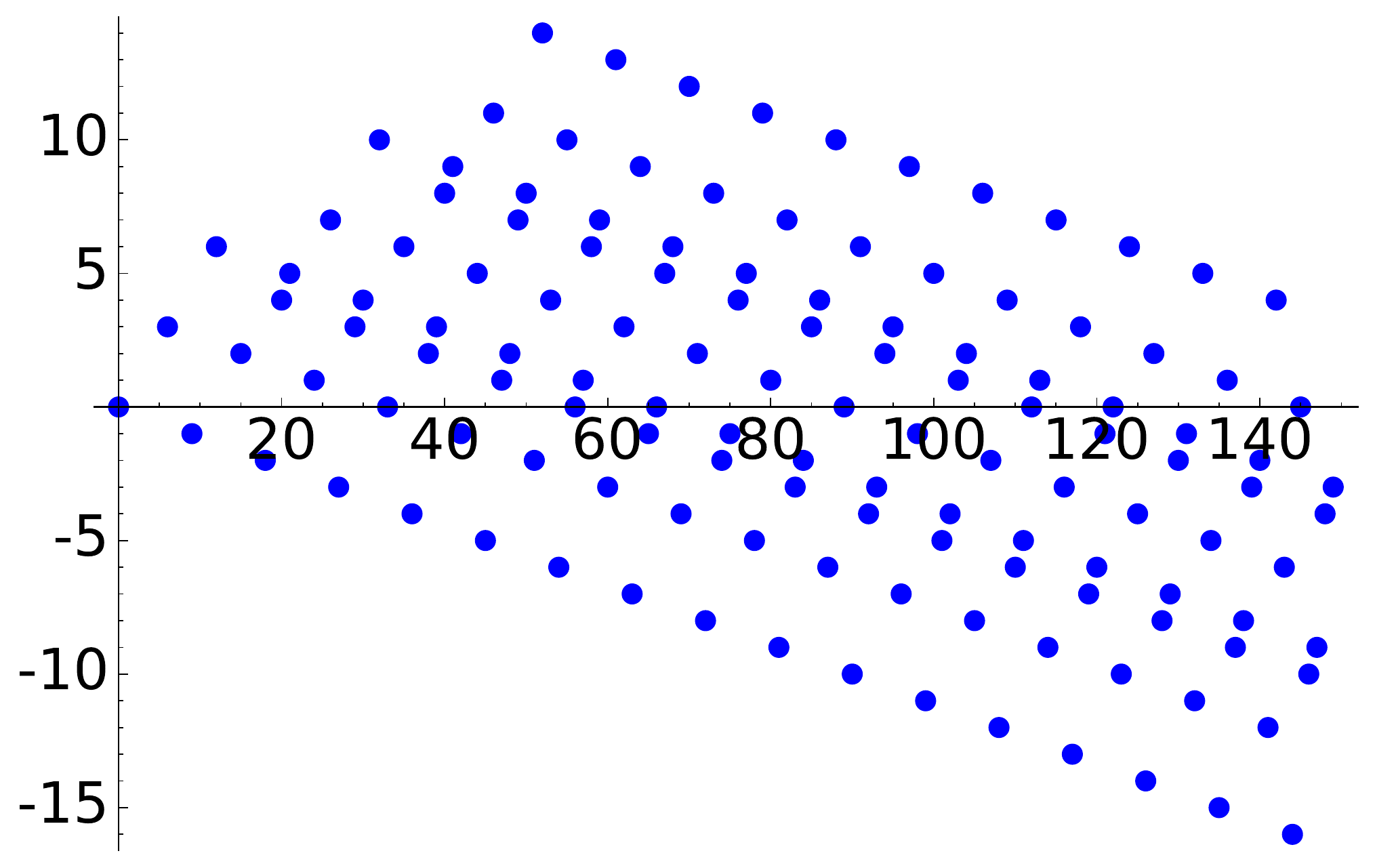}
\end{center}
\caption{Plots depicting the minimum weighted factorization lengths of elements of $S = \<6, 9, 20\>$ for the weight vectors $w = (3,1,4)$ (left) and $w = (3,-1,4)$ (right) from Example~\ref{e:weightedlength}, created using \texttt{Sage} and the \texttt{GAP} package \texttt{numericalsgps}~\cite{numericalsgpsgap}.
}
\label{f:weightedlength}
\end{figure}

\begin{lemma}[{\cite[Lemma~4.1]{elastsets}}]\label{l:oldlemma41}
Suppose $q \ge 1$, and fix $c_1, \ldots, c_r \in \ZZ$ with $r \ge q$. There exists $T \subsetneq \{1, \ldots, r\}$ satisfying $\sum_{i \in T} c_i \equiv \sum_{i=1}^r c_i \bmod q$.
\end{lemma}

\begin{lemma}\label{l:maxminquasi}
Fix a numerical semigroup $S = \<r_1, \ldots, r_k\>$, a weight vector $w \in \QQ^k$, and suppose $r_1 \le_w r_2 \le_w \cdots \le_w r_k$.  
\begin{enumerate}[(a)]
\item 
If $a \in \mathsf Z_S(n)$ satisfies $a_1 + \cdots + a_k \ge r_1$, then there is some factorization $b \in \mathsf Z_S(n)$ with $|b|_w \ge |a|_w$ and $b_1 > 0$.  

\item 
If $a \in \mathsf Z_S(n)$ satisfies $a_1 + \cdots + a_k \ge r_k$, then there is some factorization $b \in \mathsf Z_S(n)$ with $|b|_w \le |a|_w$ and $b_k > 0$.  

\end{enumerate}
\end{lemma}

\begin{proof}
First, we claim if $a' = (0, a_2', \ldots, a_k'), b' = (b_1', 0, \ldots, 0) \in \mathsf Z(n)$, then $|b'|_w \ge |a'|_w$.  Indeed, this follows from the fact that
$$|a'|_w
= \sum_{i=2}^k w_ia_i' 
= \sum_{i=2}^k \frac{w_i}{r_i} r_ia_i'
\le \sum_{i=2}^k \frac{w_1}{r_1} r_ia_i'
= \frac{w_1}{r_1} r_1b_1' 
= w_1b_1'
= |b'|_w.$$
Now, under the assumptions for part~(a), we see
$$a_1r_1 = n - a_2r_2 - \cdots - a_kr_k$$
implies $a_2r_2 + \ldots + a_kr_k \equiv n \bmod r_1$.  
Lemma~\ref{l:oldlemma41} then guarantees the existence of integers $b_2, \ldots, b_k \ge 0$ such that (i)~$b_i \le a_i$ for each $i > 1$, (ii)~$\sum_{i =2}^ka_i >\sum_{i =2}^kb_i$, and (iii)~$b_2r_2 + \cdots + b_kr_k \equiv n \bmod r_1$.
This in particular means there exists $b_1 > 0$ so that $b = (b_1,\ldots, b_k) \in \mathsf Z(n)$.  Rearranging the equation
$$n = a_1r_1 + \cdots + a_kr_k = b_1r_1 + \cdots + b_kr_k$$
yields
$$(b_1 - a_1)r_1 = (a_2 - b_2)r_2 + \cdots + (a_k - b_k)r_k$$
Applying the above claim to $(b_1-a_1, 0, \ldots, 0)$ and
$(0,a_2 - b_2, \ldots, a_k - b_k)$ implies
$$w_1(b_1 - a_1) \ge w_2(a_2 - b_2) + \cdots + w_k(a_k - b_k),$$
meaning
$$w_1b_1 + \cdots + w_kb_k \ge w_1a_1 + \cdots + w_ka_k,$$
so $|b|_w \ge |a|_w$.  This proves part~(a).  

The proof of part~(b) is analogous and thus omitted.  
\end{proof}

\begin{thm}\label{t:maxminquasi}
Fix a numerical semigroup $S = \<r_1, \ldots, r_k\>$ and a weight vector $w \in \QQ^k$, and suppose $r_1 \le_w r_2 \le_w \cdots \le_w r_k$.  Let $R = \max(r_1, \ldots, r_k)$.  
\begin{enumerate}[(a)]
\item
\label{t:maxminquasi:maxlen}
For all $n > R^2$, the maximal weighted length function $\mathsf M_w: S \to \QQ$ satisfies
$$\mathsf M_w(n) = \mathsf M_w(n - r_1) + w_1.$$

\item 
\label{t:maxminquasi:minlen}
For all $n > R^2$, the minimal weighted factorization length $\mathsf m_w: S \to \QQ$ satisfies 
$$\mathsf m_w(n) = \mathsf m_w(n - r_k) + w_k.$$

\end{enumerate}
\end{thm}

\begin{proof} 
Suppose $n > R^2$.  First, we claim there is a factorization of $n$ with maximum weighted length with positive first coordinate.  Indeed, fix any factorization $a \in \mathsf Z(n)$.  If $a_2 + \cdots + a_k < r_1$, then $a_1 > 0$ by the assumption on $n$.  On the other hand, if $a_2 + \cdots + a_k \ge r_1$ and $a_1 = 0$, then the claim follows from Lemma~\ref{l:maxminquasi}(a).  

Now, by the above claim, let $a \in \mathsf Z(n)$ denote a maximum weighted length factorization with $a_1 > 0$.  This means $a' = (a_1 - 1, a_2, \ldots, a_k) \in \mathsf Z(n - r_1)$ also has maximum weighted factorization length, so 
$$\mathsf M_w(n - r_1)
= |a '|_w 
= w_1(a_1 - 1) + w_2a_2 + \ldots w_ka_k
= |a|_w - w_1
= \mathsf M_w(n) - w_1,$$
thereby proving part~(a).  

By a similar argument, some minimal weighted length factorization $a \in \mathsf Z(n)$ has $a_k > 0$.  The proof of part~(b) then follows analogously.  
\end{proof}

\begin{remark}\label{r:wtie}
For a given numerical semigroup $S = \<r_1, \ldots, r_k\>$, if a particular weight vector $w$ induces a ``tie'' $r_1 =_w \cdots =_w r_j$ in the $w$-ordering, then Theorem~\ref{t:maxminquasi} obtains an improved period $\gcd(r_1, \ldots, r_j)$ for the quasilinear function $\mathsf M_{S,w}$.  For example, if~$S = \<6,9,10,14\>$ and $w = (2, 3, 5, 7)$, then $r_1 =_w r_2 >_w r_3 =_w r_4$, and for large $n$, $\mathsf M_{S,w}(n)$ and $\mathsf m_{S,w}(n)$ are each quasilinear with minimal periods $2$ and $3$, respectively.  See Figure~\ref{f:wtie} for a depiction.  
\end{remark}

\begin{figure}
\begin{center}
\includegraphics[width=2.8in]{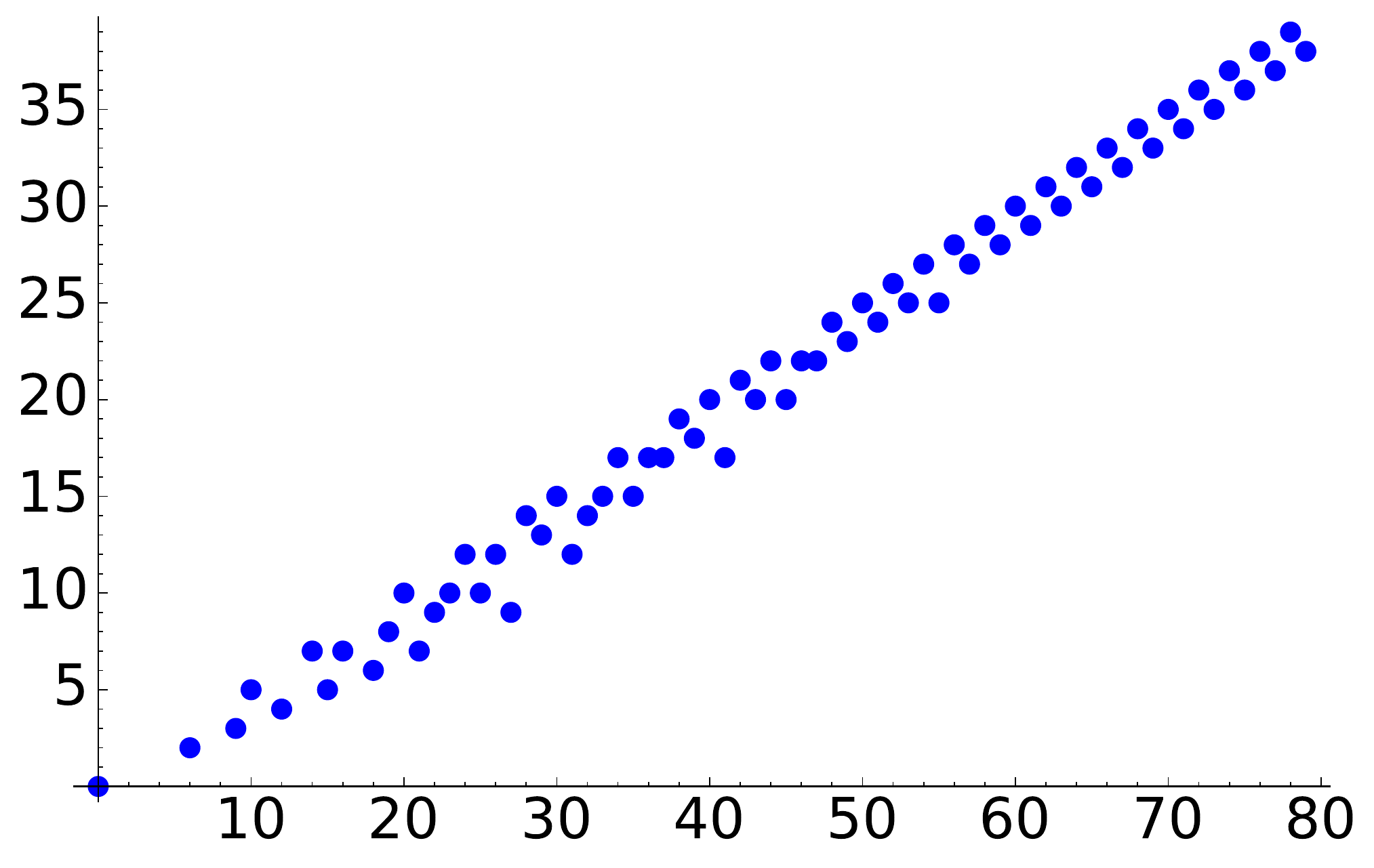}
\hspace{0.2in}
\includegraphics[width=2.8in]{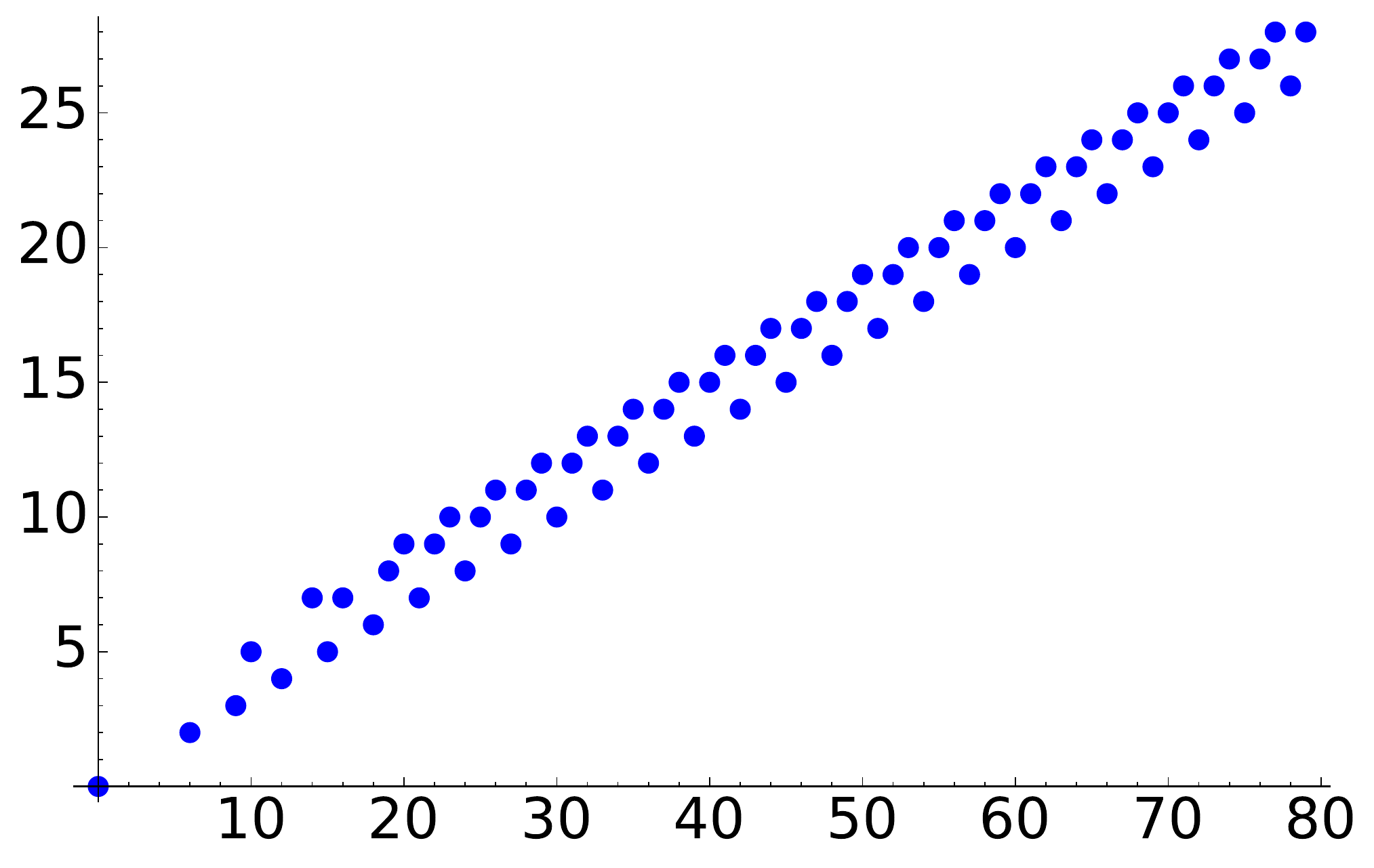}
\end{center}
\caption{Maximum (left) and minimum (right) weighted factorization lengths for $S = \<6, 9, 10, 14\>$ and $w = (2,3,5,7)$ from Remark~\ref{r:wtie}, created using \texttt{Sage} and the \texttt{GAP} package \texttt{numericalsgps}~\cite{numericalsgpsgap}.
}
\label{f:wtie}
\end{figure}

\begin{remark}\label{r:maxminwbound}
Much to our surprise, the bounds in Theorem~\ref{t:maxminquasi} do not depend on $w$, although it is worth noting that an optimal bound necessarily depends on $w$.  Indeed, suppose $S = \<9,10,23\>$.  If $w = (1,3,5)$, then $n = 64$ is the largest $n$ for which the first equality in Theorem~\ref{t:maxminquasi} fails to hold, and if $w = (6,9,5)$, then $n = 81$ is the largest such $n$.  For both weight vectors, the generator $10$ is minimal under the $w$-ordering.  
\end{remark}

The following corollary of Lemma~\ref{l:maxminquasi} and Theorem~\ref{t:maxminquasi} will be used in Section~\ref{sec:linearfamilies}.  Note the additional assumption that $w$ has positive integer entries.  

\begin{cor}\label{c:maxminquasi}
Fix a numerical semigroup $S = \<r_1, \ldots, r_k\>$, a weight vector $w \in \ZZ_{\ge 1}^k$, and suppose $r_1 \le_w r_2 \le_w \cdots \le_w r_k$.  Fix $w_0 \in \ZZ_{\ge 1}$.  
\begin{enumerate}[(a)]
\item 
If $a \in \mathsf Z_S(n)$ satisfies $a_1 + \cdots + a_k \ge w_0r_1$, then there is some factorization $b \in \mathsf Z_S(n)$ with $|b|_w - |a|_w \in w_0\ZZ_{\ge 0}$ and $b_1 > 0$.  

\item 
If $a \in \mathsf Z_S(n)$ satisfies $a_1 + \cdots + a_k \ge w_0r_k$, then there is some factorization $b \in \mathsf Z_S(n)$ with $|a|_w - |b|_w \in w_0\ZZ_{\ge 0}$ and $b_k > 0$.  

\end{enumerate}
\end{cor}

\begin{proof}
If $a_1 > 0$, then choosing $b = a$ proves part~(a), so suppose $a_1 = 0$.  Fix $a' \in \ZZ_{\ge 0}^k$ such that $a_i' \le a_i$ for each $i$ and $a_1' + \cdots + a_k' \ge r_1$, and write $n' \in S$ so that $a' \in \mathsf Z(n')$.  By Lemma~\ref{l:maxminquasi}(a), there exists $b' \in \mathsf Z(n')$ with $|b'|_w \ge |a'|_w$ and $b_1' > 0$.  If $|b'|_w = |a'|_w$, then choosing $b = b' + (a - a')$ proves part~(a), so suppose $|b'|_w > |a'|_w$.  

Now, fix a collection $c_1, \ldots, c_{w_0} \in \ZZ_{\ge 0}^k$ of vectors that sum to $a$.  
Apply the above argument to each $c_i$ (in the role of $a'$) to obtain vectors $d_1, \ldots, d_{w_0} \in \ZZ_{\ge 0}^k$ (i.e., each corresponding vector $b'$ above), and let $\ell_i = |d_i|_w - |c_i|_w$.  By Lemma~\ref{l:oldlemma41}, there exists a subset $T \subset \{1, \ldots, w_0\}$ so that $\sum_{i \in T} \ell_i \equiv 0 \bmod w_0$.  Letting 
$$b = \sum_{i \in T} d_i + a - \sum_{i \in T} c_i$$
we obtain
$$|b|_w - |a|_w = \biggl|\sum_{i \in T} d_i + a - \sum_{i \in T} c_i \biggr|_w - |a|_w = \sum_{i \in T} (|d_i|_w - |c_i|_w) \in w_0\ZZ_{\ge 0}$$
which completes the proof of part~(a).  

As in the proof of Lemma~\ref{l:maxminquasi}, the proof of part~(b) is analogous.  
\end{proof}

For the remainder of this section, we turn our attention to the weighted delta set.  As~with weighted length sets, choosing the weight vector $w = (1, \ldots, 1)$ in the following definition recovers the usual delta set.  

\begin{defn}\label{d:deltasets}
Fix a numerical semigroup $S = \<r_1, \ldots, r_k\>$, a weight vector $w \in \QQ^k$, and an element $n \in S$, and write
$$\mathsf L_{S,w}(n) = \{\ell_1 < \ell_2 < \cdots < \ell_r\}.$$
The \emph{weighted delta set} of $n$ is given by
$$\Delta_{S,w}(n) = \{\ell_i - \ell_{i-1}: i = 2, \ldots, r\},$$
and the \emph{weighted delta set} of $S$ is given by
$$\Delta_w(S) = \bigcup_{n \in S} \Delta_{S,w}(n).$$
Note that, unlike the usual delta set, it is possible to have $\Delta_w(S) = \emptyset$.  Indeed, this happens when $w_i = r_i$ for every $i$, as $\mathsf L_w(n) = \{n\}$ for every $n \in S$ in this case.  
\end{defn}

\begin{thm}\label{t:weighteddelta}
Fix a numerical semigroup $S = \<r_1, \ldots, r_k\>$ and a vector $w \in \QQ^k$.  
\begin{enumerate}[(a)]
\item 
If $\Delta_w(S) \ne \emptyset$, then $\Delta_w(S) \subset d\ZZ_{\ge 1}$, where $d = \min \Delta_w(S)$.  

\item 
We have
$$\min \Delta_w(S) = \gcd(\{w_ir_j - w_jr_i : 1 \le i < j \le r\}).$$

\item 
The set $\Delta_w(S)$ is finite.  Moreover,
$$\max \Delta_w(S) = \max_{n \in \Betti(S)} \max \Delta_w(n).$$

\end{enumerate}
\end{thm}

\begin{proof}
Each $w_i = t_i/u_i$ for some $t_i, u_i \in \ZZ$, so we must have $\Delta_w(S) \subset \delta\ZZ_{\ge 1}$, where $\delta = 1/(u_1 \cdots u_k)$.  Fix $d' \in \Delta_w(S)$, and fix $c, c' \in \ZZ_{\ge 1}$ so that $d = c\delta$ and $d' = c'\delta$.  Write $\gcd(c,c') = mc - m'c'$ for $m, m' \in \ZZ_{\ge 1}$.  We must have elements $n, n' \in S$ and factorizations $a, b \in \mathsf Z(n)$ and $a', b' \in \mathsf Z(n')$ so that $|a|_w - |b|_w = d$ and $|a'|_w - |b'|_w = d'$.  By the linearity of $|\cdot|_w$, the factorizations $ma + m'b', m'a' + mb \in \mathsf Z(mn + m'n')$ satisfy
$$|ma + m'b'|_w - |m'a' + mb|_w = m(|a|_w - |b|_w) - m'(|a'|_w - |b'|_w) = \gcd(c,c')\delta,$$
so by the minimality of $d$, we conclude $c = \gcd(c,c')$.  This proves part~(a).  

To prove part~(b), let 
$$d' = \gcd(\{w_ir_j - w_jr_i : 1 \le i < j \le r\}).$$
Since $r_je_i, r_ie_j \in \mathsf Z(r_ir_j)$, the above argument implies
$$d \mid w_ir_j - w_jr_i = |r_je_i|_w - |r_ie_j|_w,$$
meaning $d \mid d'$.  Conversely, suppose
$$a_1r_1 + \cdots + a_kr_k = b_1r_1 + \cdots + b_kr_k.$$
In order to show $d'$ divides $|a|_w - |b|_w = |a - b|_w$, by the linearity of $|\cdot|_w$ it suffices to express $a - b$ as an integer combination of the vectors $e_{ij} = r_je_i - r_ie_j$.  Notice that
$$(a_1 - b_1)r_1 = (b_2 - a_2)r_2 + \cdots + (b_k - a_k)r_k$$
and since $\gcd(r_1, \ldots, r_k) = 1$, we must have $\gcd(r_2, \ldots, r_k) \mid (a_1 - b_1)$.  As such, $a_1 - b_1 = c_2r_2 + \cdots + c_kr_k$ for some $c_i \in \ZZ$, meaning 
$$a - b - c_2e_{12} - \cdots - c_ke_{1k} = (a_2 - b_2 + c_2r_1)e_2 + \cdots + (a_k - b_k + c_kr_1)e_k,$$
which has first coordinate $0$.  Induction on $k$ concludes the proof of part~(b).  

For part~(c), fix $n \in S$ and $x, y \in \mathsf Z(n)$ where $|x|_w < |y|_w$ are sequential in $\mathsf L_w(n)$.  By~\cite[Lemma 2.1]{bfdelta}, there is a chain of factorizations $x_0, \ldots, x_t \in \mathsf Z(n)$ with $x_0 = x$, $x_t = y$, and $(x_i, x_{i+1}) = (a_i + c_i, b_i + c_i)$ for some $c_i \in \ZZ_{\ge 0}^k$ and factorizations $a_i, b_i \in \mathsf Z(n_i)$ lying in different connected components of the factorization graph~$\nabla_{n_i}$ of some Betti element~$\beta_i$.  Since $|x|_w$ and $|y|_w$ are sequential in $\mathsf L_w(n)$, there must be some $i$ so that 
$$|x_i|_w \le |x|_w < |y|_w \le |x_{i+1}|_w,$$
and no factorization $z \in \mathsf Z(\beta_i)$ can satisfy $|x|_w < |z + c_i|_w < |y|_w$.  As such, we must have $|y|_w - |x|_w \le \max \Delta_w(\beta_i)$.  This completes the proof.  
\end{proof}

\section{Linear families of numerical semigroups}
\label{sec:linearfamilies}

In the remainder of this manuscript, we examine a particular parametrized family of numerical semigroups, of the form
$$P_n := \<w_1n + r_1, \ldots, w_kn + r_k\>$$
for fixed $r = (r_1, \ldots, r_k) \in \ZZ^k$ and $w = (w_1, \ldots, w_k) \in \ZZ_{\ge 1}^k$.  The main result of this section is Theorem~\ref{t:mesalemma}, which describes the possible minimal generators that can occur for the defining toric ideal of $P_n$ for large $n$.  This result is a generalization of \cite[Theorem 3.4]{shiftyminpres}, which sat at the center of the results in \cite{shiftyminpres,shiftedaperysets} for shifted numerical semigroups (see \cite[Remark~4.10]{shiftyminpres}).  Likewise, Theorem~\ref{t:mesalemma} identifies the key structural changes that occur in $P_n$ for large~$n$ that are central to our results on Betti numbers, minimal presentations (Section~\ref{sec:minpres}) and Frobenius numbers (Section~\ref{sec:aperysets}).  

We begin by imposing some assumptions on $r_1, \ldots, r_k$ and $w_1, \ldots, w_k$, all of which can be made without loss of generality.  

\begin{notation}\label{n:parametrized}
Since $w_1, \ldots, w_k \in \ZZ_{\ge 1}$, we can reparametrize $n$ so $r_1, \ldots, r_k \in \ZZ_{\ge 0}$.  
Reorder $r_1, \ldots, r_k$ (and correspondingly $w_1, \ldots, w_k$) so that $r_1 \le_w \cdots \le_w r_k$, that is, 
$$r_1/w_1 \le \cdots \le r_k/w_k$$
(this is equivalent to Definition~\ref{d:worder} since $w$ has all positive entries).  Note that if $r_i = 0$, then $r_j = 0$ for all $j \le i$ as well.  Define 
$$W = \max\{w_1, \ldots, w_k\}
\qquad \text{and} \qquad
R = \max\{r_1, \ldots, r_k\},$$
and reparametrize $n$ appropriately so that $0 \le r_1 < w_1$.  
\end{notation}

\begin{remark}\label{r:rationalgens}
The proof of Theorem~\ref{t:mesalemma} begins by reparametrizing $P_n$ so that the first generator equals the input parameter.  However, doing so forces the constant terms $t_2, \ldots, t_k$ to be (potentially) rational.  Several times throughout the proof, Lemma~\ref{l:maxminquasi}(b) is carefully applied to the additive subsemigroup $T = \<t_2, \ldots, t_k\> \subset \QQ_{\ge 0}$ in the following sense:\ $T$ can be scaled by a unique rational value $\delta \in \QQ_{> 0}$ to obtain an isomorphic semigroup $\delta T \subset \ZZ_{\ge 0}$ with finite complement.  
\end{remark}

\begin{thm}\label{t:mesalemma}
Let $z$ and $z'$ be factorizations of a Betti element $\beta \in P_n$ in different connected components of $\nabla_\beta$ with $|z|_w > |z'|_w$.  If $n > w_1^2W\!R^2$, then 
\begin{enumerate}[(a)]
\item 
the connected components of $z$ and $z'$ in $\nabla_\beta$ contain every factorization of weighted length $|z|_w$ and $|z'|_w$, respectively; 

\item 
some factorization $y$ with $|z|_w = |y|_w$ has $y_1 > 0$; and

\item 
some factorization $y'$ with $|z'|_w = |y'|_w$ has $y_k' > 0$.  

\end{enumerate}
\end{thm}

\begin{proof}
Let $m = w_1n + r_1$ so that
$$\begin{array}{r@{}c@{}l}
P_n
&{}={}& \big\<m, \, \tfrac{w_2}{w_1}m + \big(r_2 - w_2\tfrac{r_1}{w_1}\big), \, \ldots, \, \tfrac{w_k}{w_1}m + \big(r_k - w_k\tfrac{r_1}{w_1}\big)\big\> \\[0.2em]
&{}={}& \<m, \, v_2 m + t_2, \, \ldots, \, v_k m + t_k\>,
\end{array}$$
where each $t_i = r_i - w_i\tfrac{r_1}{w_1} \ge r_i - w_i\tfrac{r_i}{w_i} = 0$ and $v_i = \tfrac{w_i}{w_1}$.  With this notation, we see $r_i \le_w r_j$ implies
$$\frac{t_i}{v_i}
= \frac{w_1r_i - w_ir_1}{w_i}
= w_1\frac{r_i}{w_i} - r_1
\le w_1\frac{r_j}{w_j} - r_1
= \frac{w_1r_j - w_jr_1}{w_j}
= \frac{t_j}{v_j}.$$
In particular, this implies (i) $t_1 = \cdots = t_{j-1} = 0$ for some $j < r$, and (ii) $t_j \le_v \cdots \le_v t_k$, viewing $v = (v_j, \ldots, v_k)$ as a weight vector for $T = \<t_j, \ldots, t_k\>$.  For simplicity, given $t \in T$ and $a \in \mathsf Z_T(t)$, we write 
$$|z|_v = v_1z_1 + \cdots + v_kz_k \qquad \text{and} \qquad |a|_v = v_ja_j + \cdots + v_ka_k$$
throughout the remainder of the proof.  The key observation is that
$$\beta - |z|_v m 
= z_1m + \sum_{i = j}^k z_i (v_i m + t_i) - z_1m - m \sum_{i = j}^k v_iz_i
= \sum_{i = j}^k z_i t_i
$$
yields a natural mapping of each factorization of $\beta \in P_n$ of weighted length $\ell$ to some factorization of $\beta - \ell m \in T$ of weighted length at most $\ell$.  Let 
$$a = (z_j, \ldots, z_k) \in \mathsf Z_T(\beta - |z |_v m) \qquad \text{and} \qquad a' = (z_j', \ldots, z_k') \in \mathsf Z_T(\beta - |z'|_v m)$$
denote the factorizations in $T$ corresponding to $z$ and $z'$, respectively.  First, we claim some factorization in the same connected component of $\nabla_{\beta}$ as $z'$ has positive last coordinate. If $z_k' > 0$, then the claim is proven, so suppose $z'_k = 0$.  
Since $w \in \ZZ_{\ge 1}^k$,
$$\begin{array}{r@{}c@{}l}
\beta - |z'|_v m
&{}={}& \displaystyle
\beta - \tfrac{1}{w_1}|z'|_w m
\ge \beta - \tfrac{1}{w_1}\big(|z|_w - 1\big)m
= \tfrac{1}{w_1}m + \beta - |z|_v m
\ge \displaystyle
\tfrac{1}{w_1}m
\ge n.
\end{array}$$
By assumption, $n > w_1^2\!R^2$, so writing $\delta \in \QQ_{> 0}$ for the unique rational value such that $\delta T \subset \ZZ_{\ge 0}$ has finite complement, this implies 
$$a_j' + \cdots + a_k' \ge \tfrac{1}{R}(a_j't_j + \cdots + a_k't_k) = \tfrac{1}{R}(\beta - |z'|_v m) \ge \tfrac{1}{R}n > \tfrac{1}{R}w_1^2R^2 \ge w_1\delta t_k,$$
and thus Corollary~\ref{c:maxminquasi}(b) implies some factorization with positive last coordinate and weighted length having integer difference from $|a'|_v$ can be obtained from $a'$ by replacing all but at least one generator with copies of $t_k$.  In particular, this factorization is in the same connected component as $a'$.  Moreover, Corollary~\ref{c:maxminquasi}(b) implies some factorization $a'' \in \mathsf Z_T(\beta - |z'|_v m)$ whose weighted length is minimal among those satisfying $|a'|_v - |a''|_v \in \ZZ$ has $a_k'' > 0$.  
Under the above factorization mapping, the factorization $z'' = (|z'|_v - |a''|_v, 0, \ldots, 0, a_j'', \ldots, a_k'') \in \mathsf Z_{P_n}(\beta)$ corresponds to $a''$ since
$$(|z'|_v - |a''|_v)m + \sum_{i = j}^k a_i''(v_i m + t_i) = (|z'|_v - |a''|_v)m + |a''|_v m + (\beta - |z'|_v m) = \beta.$$
The factorization $z''$ is thus in the same connected component of $\nabla_{\beta}$ as $z'$ and has $z_k'' > 0$, so the claim is proved.  

Since $z$ and $z'$ are in different connected components of $\nabla_{\beta}$, we must have $z_k = 0$.  This means $a_j + \cdots + a_k < w_1 \delta t_k$, as otherwise the above argument would yield a factorization of $\beta$ with positive last coordinate that is connected to $z$ in $\nabla_{\beta}$.  Writing $V = \max(v_1, \ldots, v_k) = W/w_1$, the assumption $n > w_1^2W\!R^2$ implies
$$\begin{array}{r@{}c@{}l}
|z|_v 
&{}>{}& \displaystyle |z'|_v \ge |a'|_v = \sum_{i = j}^k v_ia_i' = \sum_{i = j}^k \frac{v_i}{t_i}a_i't_i \ge \frac{v_k}{t_k}\sum_{i = j}^k a_i't_i = \frac{v_k}{t_k}(\beta - |z'|_v m) > \frac{v_k}{t_k} w_1^2W\!R^2 \\
&{}={}& \displaystyle \frac{w_k}{t_k}w_1W\!R^2 \ge w_1W\!R \ge w_1t_kW = w_1^2t_kV \ge w_1\delta t_kV > V \sum_{i = j}^k a_i \ge \sum_{i = j}^k v_ia_i = |a|_v,
\end{array}$$
so the factorization $y = (|z|_v - |a|_v, 0, \ldots, 0, a_j, \ldots, a_k)$ proves part~(b).  
Additionally, either $y$ is connected to $z$ in $\nabla_\beta$, or $z_1 = 0$ and $z_j = \cdots = z_k = 0$.  In the latter case, the preceeding inequalities imply $y_1 = |z|_v > w_kW\!R \ge W$, so one of the factorizations 
$$y - w_i e_1 + w_1 e_i \in \mathsf Z_{P_n}(\beta) \qquad \text{for} \qquad 1 \le i \le j - 1$$
yields a path from $y$ to $z$ in $\nabla_\beta$ since one of the values $z_2, \ldots, z_{j-1}$ must be positive.  This proves the first half of part~(a).  

Lastly, suppose $z_k' = 0$.  The above argument yielded a factorization $z''$ in the same connected component as $z'$ with $z_k'' > 0$ and corresponding factorization $a''$ having minimal weighted length.  Since $z$ and $z'$ are in different connected components of~$\nabla_\beta$, the first half of part~(a) implies $z_1' = z_1'' = \cdots = z_{j-1}' = z_{j-1}'' = 0$ and thus $|a'|_v = |a''|_v$.  
This proves part~(c), and applying the arguments thus far to any factorization of weighted length $|z'|_w$ yields a path in $\nabla_\beta$ to $z'$ through $z''$.  This completes the proof.  
\end{proof}

\begin{example}\label{e:positivecoords}
In Theorem~\ref{t:mesalemma}(c), we cannot ensure that all choices of factorizations $z$ and $z'$ has positive first and last coordinates, respectively.  Indeed, if $r = (0,0,2,3)$ and $w = (5,7,2,3)$, then $\beta = 1980$ is a Betti element of $P_{44}$ with 
$$(0, 0, 22, 0), \quad (0, 0, 19, 2), \quad (0, 0, 9, 0), \quad \text{and} \quad (0, 0, 2, 5)$$
among its factorizations.  The key is that in the proof of Theorem~\ref{t:mesalemma}, there are ties in the $w$-ordering for both first and last place.  As a consequence of Theorem~\ref{t:minpresmap}, this phenomenon also occurs for a Betti element of $P_{44+15m}$ for each $m \ge 0$.  
\end{example}

\begin{example}\label{e:uglymapping}
In the proof of Theorem~\ref{t:mesalemma}, the natural mapping from factorizations of $\beta \in P_n$ of weighted length $\ell$ to factorizations of $\beta - \ell m \in T$ of weighted length at most $\ell$ need not be injective nor surjective.  Let $r = (0,0,5,7,9)$ and $w = (2,3,5,7,8)$.  Certainly, the factorizations $(8,0,0,0,0)$ and $(2,4,0,0,0)$ of $\beta = 704 \in P_{44}$ are mapped to the same factorization of $0 \in T = \<\frac{5}{2}, \frac{7}{2}, \frac{9}{2}\>$.  What is perhaps more subtle is that $\beta = 1620$ has 2 factorizations, namely
$$\mathsf Z_{P_{44}}(1620) = \{(0, 0, 3, 3, 0), (2, 0, 0, 0, 4)\}$$
but the corresponding element $18 \in T$ has factorizations
$$\mathsf Z_T(18) = \{(3, 3, 0), (4, 1, 1), (0, 0, 4)\}$$
and the second does not correspond to any factorizations of $\beta$.  The issue is that $v = (1,\frac{3}{2},\frac{5}{2},\frac{7}{2},4)$, so $a = (4,1,1)$ has non-integral weighted length $|a|_v = \frac{35}{2}$, so it is impossible to fill the first or second coordinates of a corresponding factorization of $\beta$ to obtain the necessary weighted length.  This is why, when constructing factorizations of $\beta$ from factorizations of elements of $T$ at several locations in the proof of Theorem~\ref{t:mesalemma}, we must ensure that the first coordinate (a weighted length difference) is integral.  
\end{example}


Now we can state a generalization of \cite[Corollary 3.5]{shiftyminpres} and \cite[Corollary 5.7]{shiftyminpres} which follows from Theorem~\ref{t:mesalemma}.  

\begin{cor}\label{c:bettidelta}
If $n > w_1^2W\!R^2$, then $\Delta_w(P_n) = \{d\}$, where
$$d = 
\gcd(w_1, \ldots, w_{j-1}, \min \Delta_w(S))\gcd(S)
$$
with $r_{j-1} = 0 < r_j$ and $S = \<r_j, \ldots, r_k\>$.  
\end{cor}

\begin{proof}
Since 
$$w_{i'}(w_in + r_i) - w_{i}(w_{i'}n + r_{i'}) = w_{i'}r_i - w_ir_{i'}$$
for any $i, i' \le k$, 
we have
$$\begin{array}{r@{}c@{}l}
\min\Delta_w(P_n)
&{}={}& \gcd(\{w_ir_{i'} : 1 \le i < j \le i' \le k\} \cup \{w_ir_{i'} - w_{i'}r_i : j \le i < i' \le k\}) \\
&{}={}& \gcd(w_1, \ldots, w_{j-1}, \min \Delta_w(S))\gcd(r_1, \ldots, r_k),
\end{array}$$
so the first claim follows from Theorem~\ref{t:weighteddelta}(b).  

Applying Theorem~\ref{t:weighteddelta}(c), we will show if two factorizations $z, z'\in \mathsf Z(m)$ satisfy $|z|_w - |z'|_w \ge 2d$, then $z$ and $z'$ must be in the same connected component of $\nabla_m$.  Let $\ell = |z|_w - |z'|_w$.  Just as in the proof of Theorem~\ref{t:mesalemma}, we know
$$m - |z|_wn = z_1r_1 + \cdots + z_kr_k \in S,$$
so since $n > w_1^2W\!R^2$, we have 
$$(m - |z|_wn) + (\ell - d)n
= (m - |z'|_w n - \ell n) + (\ell - d)n 
= m - (|z'|_w + d)n \in S.$$
Any factorization of the above element of $S$ corresponds to a factorization $z'' \in \mathsf Z(m)$ with $|z''|_w = |z'|_w + d$ that is connected to both $z$ and $z'$ in $\nabla_m$ by Theorem~\ref{t:mesalemma}.  
\end{proof}



\section{Minimal presentations of parametrized semigroups}
\label{sec:minpres}

Let $\pi_n:\ZZ_{\ge 0}^k \to P_n$ denote the map
$$\pi_n(z) = \sum_{i = 1}^k z_i(w_in + r_i) = |z|_w n + \sum_{i = 1}^k z_ir_i,$$
called the \emph{factorization homomorphism} of $P_n$.  The equivalence relation $\ker \pi_n$  on $\ZZ_{\ge 0}^k$, called the \emph{kernel congruence}, is given by $(z,z') \in \ker\pi_n$ whenever $\pi_n(z) = \pi_n(z')$, (that is, when $z$ and $z'$ are factorizations for the same element in $P_n$).  Here, $\ker \pi_n$ is a congruence since it is closed under \emph{translation}, that is, $(z + u, z' + u) \in \ker \pi_n$ for every $(z,z') \in \ker\pi_n$ and $u \in \ZZ_{\ge 0}^k$.  

A minimal presentation (Definition~\ref{d:minpres}) of a given semigroup $T$ encodes a particular choice of minimal relations (or \emph{trades}) between the generators of $T$.  
They are one of the fundamental tools with which to study the factorization structure of numerical semigroups, and are closely connected to the defining toric ideal of $T$ (Remark~\ref{r:minprestoricideal}).   
For a thorough introduction, we refer the reader to \cite[Chapter~9]{fingenmon} and \cite[Chapter~7]{numerical}.  

The results in this section generalize those in \cite{shiftyminpres}, where a special (unweighted) case of the parametrization defining $P_n$ is considered.  
At the heart of the main results in~\cite{shiftyminpres} is a map between kernel congruences, used to establish a correspondence between minimal presentations for large $n$ that restricts to a bijection on Betti elements.  Our analogous map, $\Phi_n$, is defined in Proposition~\ref{p:mapwelldefined}, and its key properties (which closely mirror those in \cite{shiftyminpres}) are given in Proposition~\ref{p:mapproperties}.  The main results are Theorem~\ref{t:minpresmap} and Corollary~\ref{c:bettiquasi}, which establish periodicity results for the minimal presentations and Betti elements of $P_n$, respectively, for large $n$.  For our more general parametrization, the period turns out to be
$$p = w_1r_k - w_kr_1,$$
which specializes to a period of $r_k$ when $w_1 = 1$ and $r_1 = 0$ (as in \cite{shiftyminpres}).  


In this section, we omit several proofs that are nearly identical to those in \cite{shiftyminpres}, including only those aspects that are different in our more general setting.  

\begin{defn}\label{d:minpres}
Fix a numerical semigroup $T = \<t_1, \ldots, t_k\>$ and let $\pi:\ZZ_{\ge 0}^k \to T$ denote the factorization homomorphism of $T$.  A \emph{presentation} for $T$ is a set of relations $\rho \subset \ker \pi$ such that $\ker \pi$ is the unique minimal (w.r.t.\ containment) congruence on $\ZZ_{\ge 0}^k$ containing $\rho$.  Equivalently, between any two factorizations $(z, z') \in \ker \pi$, there exists a \emph{chain} $a_0, a_1, \ldots, a_r$ with $a_0 = z$, $a_r = z'$, and 
$$(a_{i-1},a_i) = (b_i,b_i') + (u_i,u_i) \in \ker \pi$$
for some $(b_i, b_i') \in \rho$ and $u_i \in \ZZ_{\ge 0}^k$ for each $i \le r$.  We say $\rho$ is \emph{minimal} if it is minimal with respect to containment among all presentations of $T$.  
\end{defn}

\begin{remark}\label{r:minprestoricideal}
Returning to the commutative algebra viewpoint in Remark~\ref{r:bettinumbers}, minimal presentations encode minimal generating sets of toric ideals.  Let $T = \<t_1, \ldots, t_k\>$, and write $I = \ker\varphi$ for the defining toric ideal of $T$, where $\varphi$ is the ring homomorphism 
$$\begin{array}{r@{}c@{}l}
\varphi:\CC[x_1, \ldots, x_k] &{}\to{}& \CC[y] \\
x_i &{}\mapsto{}& y^{t_i}.
\end{array}$$
Each relation $(a, b) \in \ker\pi$ corresponds to a binomial 
$$x_1^{a_1} \cdots x_k^{a_k} - x_1^{b_1} \cdots x_k^{b_k} \in I,$$
and each minimal presentation of $T$ corresponds to some minimal generating set of $I$.  As an example, if $T = \<6, 9, 20\>$, then the minimal presentations of $T$ are
$$\begin{array}{ll}
\{((3,0,0), (0,2,0)), ((10,0,0), (0,0,3))\}, & \{((3,0,0), (0,2,0)), ((7,2,0), (0,0,3))\}, \\
\{((3,0,0), (0,2,0)), ((\phantom{0}4,4,0), (0,0,3))\}, & \{((3,0,0), (0,2,0)), ((1,6,0), (0,0,3))\}, 
\end{array}$$
each of which corresponds to one of the 4 minimal generating sets of the defining toric ideal $I \subset \CC[x,y,z]$ listed in Remark~\ref{r:bettinumbers}.  
\end{remark}

\begin{example}\label{e:minpresmap}
Let $r = (1,2,4,6)$ and $w = (3,4,6,9)$, and consider the following minimal presentations for $P_n$ with $n$ identical modulo $p = 3 \cdot 6 - 9 \cdot 1 = 9$.  
$$
\begin{array}{
r@{\,\,\,}
l@{}r@{\,\,}r@{\,\,}r@{\,\,}r@{\,}r@{\,\,}r@{\,\,}r@{\,}r@{}l@{\,\,\,\,}
l@{}r@{\,\,}r@{\,\,}r@{\,\,}r@{\,}r@{\,\,}r@{\,\,}r@{\,}r@{}l@{\,\,\,\,}
}
P_{506}: 
& (( & 0, & 0, & 3, & 0), & (0, & 0, & 0, & 2 & )), 
& (( & 0, & 3, & 0, & 0), & (2, & 0, & 1, & 0 & )), 
\\
& (( & 506, & 1, & 0, & 0), & (0, & 0, & 0, & 169 & )), 
& (( & 508, & 0, & 0, & 0), & (0, & 2, & 2, & 167 & ))
\\[0.5em]

P_{515}: 
& (( & 0, & 0, & 3, & 0), & (0, & 0, & 0, & 2 & )), 
& (( & 0, & 3, & 0, & 0), & (2, & 0, & 1, & 0 & )), 
\\
& (( & 515, & 1, & 0, & 0), & (0, & 0, & 0, & 172 & )), 
& (( & 517, & 0, & 0, & 0), & (0, & 2, & 2, & 170 & ))
\\[0.5em]

P_{524}: 
& (( & 0, & 0, & 3, & 0), & (0, & 0, & 0, & 2 & )), 
& (( & 0, & 3, & 0, & 0), & (2, & 0, & 1, & 0 & )), 
\\
& (( & 524, & 1, & 0, & 0), & (0, & 0, & 0, & 175 & )), 
& (( & 526, & 0, & 0, & 0), & (0, & 2, & 2, & 173 & ))

\end{array}
$$
Each first-row relation $(z,z')$ satisfies $|z|_w = |z'|_w$, and each second-row relation $(z,z')$ satisfies $|z| = |z'| + 1$.  In the latter case, each time $n$ is increased by $p = 9$, the value of $z_1$ increases by $w_4 = 9$ and $z_4'$ increases by $w_1 = 3$.  
\end{example}

\begin{defn}\label{d:monotonechain}
A chain $a_0, a_1, \ldots, a_r$ of factorizations is \emph{$w$-monotone} if the sequence $|a_0|_w, |a_1|_w, \ldots, |a_r|_w$ is monotone.  
\end{defn}

\begin{prop}\label{p:mapwelldefined}
The map $\Phi_n \colon \ker \pi_n \to \ker \pi_{n + p}$ given by
$$\Phi_n(z,z')
= \left\{\begin{array}{ll}
(z + \ell w_k e_1, z' + \ell w_1 e_k) & \text{if } |z|_w > |z'|_w \\
(z + \ell w_1 e_k, z' + \ell w_k e_1) & \text{if } |z|_w < |z'|_w \\
(z,z')                                & \text{if } |z|_w = |z'|_w
\end{array}\right.
$$
for $(z,z') \in \ker \pi_n$ and  $\ell = \big| |z|_w - |z'|_w \big|$ is well defined.
\end{prop}

\begin{proof}
Fix $(z,z') \in \ker \pi_n$ with $z = (z_1, \dots, z_k)$ and $z' = (z_1', \dots, z_k')$.  By symmetry, we can assume that $\ell = |z|_w - |z'|_w \ge 0$.  Now, we simply use $\pi_n(z) = \pi_n(z')$ to verify 
$$\begin{array}{r@{}c@{}l}
\displaystyle \pi_{n + p}(z + \ell w_ke_1)
&{}={}& \displaystyle
\ell w_k r_1 + \sum_{i = 1}^k z_i(w_i(n + p) + r_i)
= \pi_n(z) + |z|_wp + \ell w_k r_1
\\ &{}={}& \displaystyle
\pi_n(z') + |z'|_wp + \ell(p + w_k r_1)
= \pi_n(z') + |z'|_wp + \ell w_1 r_k
\\ &{}={}& \displaystyle
\ell w_1 r_k + \sum_{i = 1}^k z_i'(w_i(n + p) + r_i)
= \pi_{n + p}(z' + \ell w_1e_k),
\end{array}$$
as desired.  
\end{proof}

\begin{prop}\label{p:mapproperties}
Fix $n \in \ZZ_{\ge 1}$, $\rho \subset \ker \pi_n$, and $(z, z') \in \rho$, and let $(y, y') = \Phi_n(z, z')$.  
\begin{enumerate}[(a)]
\item
\label{p:mapproperties_injective}
The map $\Phi_n$ is injective.  

\item 
\label{p:mapproperties_lendiffs}
The map $\Phi_n$ preserves weighted length differences: $|z|_w - |z'|_w = |y|_w - |y'|_w$.  

\item 
\label{p:mapproperties_closures}
The map $\Phi_n$ preserves the reflexive, symmetric, and translation closure operations: if $\rho$ is reflexive, symmetric, or closed under translation, then so is $\Phi_n(\rho)$.  

\item 
\label{p:mapproperties_monotonechains}
The map $\Phi_n$ preserves $w$-monotone chain connectivity: if $\rho$ is translation-closed and there exists a $w$-monotone $\rho$-chain from $z$ to $z'$, then there exists a $w$-monotone $\Phi_n(\rho)$-chain from $y$ to $y'$.  

\end{enumerate}
\end{prop}

\begin{proof}
The proof is nearly identical to that of \cite[Proposition~4.4]{shiftyminpres}.  
\end{proof}

%

Just as in \cite{shiftyminpres}, the main obstruction to Theorem~\ref{t:minpresmap} for arbitrary $n$ is that $\Phi_n$ need only preserve connectivity by $w$-monotone chains.  By Proposition~\ref{p:monotonechain}, any two factorizations $(z,z') \in \ker \pi_n$ are guaranteed to be connected by a $w$-monotone chain if $n$ is large enough.  

\begin{prop}\label{p:monotonechain}
Fix $n > w_1^2W\!R^2$ and a minimal presentation $\rho \subseteq \ker \pi_n$.  There exists a $w$-monotone $\rho$-chain between any $(z,z') \in \ker \pi_n$.
\end{prop}

\begin{proof}
Without loss of generality, assume $\gcd(z,z') = 0$.  Since $\Cong(\rho) = \ker \pi_n$, there exists a chain $z = a_0, a_1, \ldots, a_r = z'$ of factorizations such that for each $i < r$, we have 
$$(a_i,a_{i+1}) = (b_i,b_i') + (u_i,u_i), \qquad (b_i, b_i') \in \rho, u_i \in \ZZ_{\ge 0}^k,$$
where $b_i$ and $b_i'$ occur in distinct connected components of the graph $\nabla\!_\beta$ of $\beta = \pi_n(b_i)$.  By Corollary~\ref{c:bettidelta}, $\Delta_w(P_n) = \{d\}$, so $|a_i|_w - |a_{i-1}|_w \in \{-d,0,d\}$ for each $i \le r$.  

By induction on the chain length $r$, it suffices to assume $|a_1|_w = \cdots = |a_{r-1}|_w$ and $|z|_w = |a_0|_w = |a_r|_w = |z'|_w$, and prove there exists a weighted length preserving chain from $z$ to $z'$.  Indeed, any non-monotone chain must contain such a subchain, which could then be ``flattened'' to a subchain with all equal weighted lengths.  

First, suppose $|a_1|_w = |a_0|_w + d$.  Applying Theorem~\ref{t:mesalemma} to $(b_0,b_0')$ and $(b_{r-1},b_{r-1}')$, we see that $z$ and $z'$ share support with some factorizations $y$ and $y'$, respectively, with positive last coordinates and weighted lengths equal to $|z|_w = |z'|_w$.  By induction on the semigroup element $\pi_n(z)$, there exist weighted length preserving chains connecting $z$ to $y$, $y$ to $y'$, and $y'$ to $z'$.  The case $|a_1|_w = |a_0|_w - d$ follows similarly.  
\end{proof}

\begin{thm}\label{t:minpresmap}
For any $n > w_1^2W\!R^2$, the image of any minimal presentation $\rho$ of $P_n$ under the map $\Phi_n:\ker \pi_n \to \ker \pi_{n + p}$ is a minimal presentation of $P_{n + p}$.  
\end{thm}

\begin{proof}
We must show that any minimal presentation $\rho \subset \ker \pi_n$ of $P_n$ satisfies
$$\Cong(\Phi_n(\rho)) = \ker \pi_{n + p},$$
that is, the image of $\rho$ under $\Phi_n$ is a presentation for $P_{n + p}$.  Fix $(y,y') \in \ker \pi_{n + p}$, and let $m = \pi_{n + p}(y)$.  By Proposition~\ref{p:monotonechain}, there exists a $w$-monotone $\rho$-chain from $y$ to $y'$, which we can assume is $w$-monotone decreasing by Proposition~\ref{p:mapproperties}\eqref{p:mapproperties_closures}.  We can also assume each step in this chain has the form $(b,b') + (u,u)$ for some $u \in \ZZ_{\ge 0}^k$ and $b, b' \in \mathsf Z(\beta)$ lying in different connected components of $\nabla\!_\beta$.  By Proposition~\ref{p:mapproperties}\eqref{p:mapproperties_closures}, it suffices to prove each $(b,b')$ lies in $\Cong(\Phi_n(\rho))$, so we can assume $y$ and $y'$ lie in different connected components of $\nabla\!_m$.  

First, if $|y|_w = |y'|_w$, then $\Phi_n(y,y') = (y,y')$ by Proposition~\ref{p:mapwelldefined}, so applying $\Phi_n$ to any $w$-monotone (in this case, length preserving) $\rho$-chain from $y$ to $y'$ yields a $\Phi_n(\rho)$-chain from $y$ to $y'$ by Proposition~\ref{p:mapproperties}\eqref{p:mapproperties_monotonechains}.  In particular, $\Phi_n(\rho)$ connects any two factorizations of equal weighted length.  On the other hand, if $|y|_w > |y'|_w$, then Corollary~\ref{c:bettidelta} implies $|y|_w = |y'|_w + d$, where $\Delta_w(P_n) = \{d\}$.  By Theorems~\ref{t:mesalemma}(b) and~(c), some factorizations $x$ and $x'$ in the same connected components as $y$ and $y'$, respectively, satisfy $x_1 \ge d$ and $x_k' \ge d$.  Since
$$\Phi_n(x - dw_ke_1, x' - dw_1e_k) = (x,x'),$$
$\Phi_n(\rho)$ connects $x$ and $x'$ in $\nabla_m$.  As Theorem~\ref{t:mesalemma}(a) implies there are length-preserving chains from $y$ to $x$ and from $x'$ to $y'$, Proposition~\ref{p:mapproperties}\eqref{p:mapproperties_monotonechains} completes the proof.  
\end{proof}

We are now ready to prove that the Betti elements of $P_n$ are eventually periodic.  

\begin{cor}\label{c:bettiquasi}
For $n > w_1^2W\!R^2$, the map $\varphi_n: \Betti(P_n) \to \Betti(P_{n+p})$ given by
$$\beta \mapsto \left\{\begin{array}{@{\,}ll}
\beta + \lambda p & \text{if } \mathsf L_{P_n,w}(\beta) = \{\lambda\} \\
\beta + \lambda p + dw_1(w_kn + r_k) & \text{if } \mathsf L_{P_n,w}(\beta) = \{\lambda, \lambda + d\}
\end{array}\right.$$
is a bijection, where $\Delta_w(P_n) = \{d\}$.  
\end{cor}

\begin{proof}
Fix $z, z' \in \mathsf Z_{P_n}(\beta)$ with $|z|_w = \max\mathsf L_{P_n,w}(\beta)$ and $|z'|_w = \min\mathsf L_{P_n,w}(\beta)$.  In the first case, $\lambda = |z|_w = |z'|_w$, so $\Phi_n(z,z') = (z,z')$, so Theorem~\ref{t:minpresmap} implies 
$$\varphi_n(\beta) = \beta + \lambda p = \sum_{i = 1}^k z_i(w_in + r_i) + \sum_{i = 1}^k w_iz_i p = \sum_{i = 1}^k z_i(w_i(n + p) + r_i) \in \Betti(P_{n+p}).$$
In the second case, $\lambda = |z|_w - d = |z'|_w$ by Corollary~\ref{c:bettidelta}, so $\Phi_n(z,z') = (z + de_1, z' + de_k)$, and thus Theorem~\ref{t:minpresmap} implies
$$\varphi_n(\beta) = \beta + \lambda p + dw_1(w_kn + r_k) = dw_1(w_kn + r_k) + \sum_{i = 1}^k z_i'(w_i(n + p) + r_i) \in \Betti(P_{n+p}).$$
As such, $\varphi_n$ is a bijection by Theorem~\ref{t:minpresmap}.  This completes the proof.  
\end{proof}

\begin{remark}\label{r:bettidichotomy}
Just as in \cite{shiftyminpres}, Corollary~\ref{c:bettiquasi} implies the elements of $\Betti(M_n)$ fall into two distinct categories:\ those with minimal relations of equal length (which increase linearly in $n$ upon successive applications of $\Phi_n$), and those with minimal relations of different length (which increase quadratically in $n$ upon successive applications of $\Phi_n$).  

As an additional consequence of Corollary~\ref{c:bettiquasi}, we see the function $n \mapsto |\!\Betti(P_n)|$ is $p$-periodic for $n > w_1^2W\!R^2$, including if the elements of $\Betti(P_n)$ are counted with multiplicity (that is,~if each element $\beta \in \Betti(P_n)$ appears once for each relation between factorizations of $\beta$ occuring in a minimal presentation for $P_n$).  In the commutative algebra language of Remarks~\ref{r:bettinumbers} and~\ref{r:minprestoricideal}, this says the number of minimal generators of the defining toric ideal of $P_n$ is $p$-periodic in $n$.  
\end{remark}

\section{Apery sets and the Frobenius number}
\label{sec:aperysets}

In the final section of this paper, we examine the Frobenius number (Corollary~\ref{c:frobquasi}), genus (Corollary~\ref{c:genusquasi}), and type (Corollary~\ref{c:pseudofrobshifted}) of $P_n$ for large $n$.  Our results utilize the Ap\'ery set (Definition~\ref{d:apery}) of $P_n$, from which each of these quantities can be quickly obtained (indeed, in numerical semigroup computations, one often computes the Ap\'ery set first since doing so has roughly the same computational complexity).  Most of the results in this section generalize those in \cite{shiftedaperysets}.  

Throughout this section, we add the following assumptions on the parametrization of $P_n$; the difficulties in the general case are discussed in Remark~\ref{r:aperysetrestriction}.  

\begin{notation}\label{n:parametrizedapery}
Throughout this section, we restrict to the case $w_1 = 1$ (and consequently $r_1 = 0$), so that 
$$P_n = \<n, w_2n + r_2, \ldots, w_kn + r_k\>,$$
in addition to all existing assumptions from Notation~\ref{n:parametrized}.  Moreover, let $S = \<r_2, \ldots, r_k\>$ and $d = \gcd(S)$.  
\end{notation}

\begin{defn}\label{d:apery}
Fix an additive subsemigroup $T \subset (\ZZ_{\ge 0}, +)$, and let $d = \gcd(T)$.  The~\emph{Ap\'ery set} of $m \in T$ is
$$\Ap(T; m) = \{t \in T : t - m \in \ZZ \setminus T\}.$$
The \emph{genus} of $T$ is the number $\mathsf g(T) = |d\ZZ_{\ge 0} \setminus T|$ of positive integer multiples of $d$ lying outside of $T$, and the \emph{Frobenius number} of $T$ is the largest integer multiple of $d$ outside of $T$, that is, $\mathsf F(T) = \max(d\ZZ_{\ge 0} \setminus T)$.  
\end{defn}

\begin{remark}\label{r:aperyfacts}
The quantities in Definition~\ref{d:apery} are usually only defined for numerical semigroups (that is, in the case when $d = 1$).  We will make use here of the following properties of the Ap\'ery set, each of which follows immediately from a known result in the usual setting~\cite{numerical}.  
\begin{enumerate}[(a)]
\item 
Each element of $\Ap(T; m)$ is distinct modulo $m$.  In particular, $\left|\Ap(T; m)\right| = m/d$.  

\item 
We have
$$\mathsf F(S) = \max(\Ap(T; m)) - m
\qquad \text{and} \qquad 
\mathsf g(S) = \sum_{t \in \Ap(T; m)} \bigg\lfloor \frac{t}{m} \bigg\rfloor,
$$
known in the literature as Selmer's formulas \cite{selmersformula}.  

\end{enumerate}
\end{remark}

\begin{example} \label{e:apery}
If $T = \<6, 9, 20\>$, then $\Ap(T; 6) = \{0, 49, 20, 9, 40, 29\}$, where the elements are listed based on their equivalence class modulo $6$.  From Selmer's formulas in Remark~\ref{r:aperyfacts}, we conclude $\mathsf F(S) = 43$ and $\mathsf g(S) = \tfrac{147}{6} - \tfrac{5}{2} = 22$.  
\end{example}

\begin{thm}\label{t:parametrizedapery}
If $n > W\!R^2$, then
$$\Ap(P_n;n) = \{i + \mathsf m_{S,w}(i)n \mid i \in \Ap(S;dn)\}.$$
Moreover, we have
$$\mathsf L_{P_n,w}(i + \mathsf m_{S,w}(i)n) = \{\mathsf m_{S,w}(i)\}$$
for each $i \in \Ap(S;dn)$.  
\end{thm}

\begin{proof}
Fix $a \in \Ap(n)$.  Since $a - n \notin P_n$, no factorization of $a$ has positive first coordinate, so Theorem~\ref{t:mesalemma} implies $\mathsf L_{P_n,w}(a) = \{\ell\}$ for some $\ell \in \ZZ_{\ge 0}$.  Let $i = a - \ell n$.  

The proof of Theorem~\ref{t:mesalemma} establishes a natural mapping
$$\begin{array}{rcl}
\{z\in \mathsf Z_{P_n}(a) : |z|_w = \ell\} & \rightarrow & \{s \in \mathsf Z_{S}(a - \ell n): |s|_w \le \ell\} \\
(z_0,z_1, \ldots, z_k) & \mapsto & (z_1, \ldots, z_k)
\end{array}$$
between factorizations of $a \in P_n$ and factorizations of $i = a - \ell n \in S$ of weighted length at most $\ell$.  Moreover, in this setting, the above map is a bijection, since for any factorization $(z_2, \ldots, z_k) \in \mathsf Z_S(i)$,  letting $z_1 = \ell - |z|_w$ yields a factorization $(z_1, \ldots, z_k) \in \mathsf Z_{P_n}(a)$.  Since $a \in \Ap(P_n; n)$, no factorization of $a$ has positive first coordinate, so we must have 
$$\ell = |z|_w = \mathsf m_{S,w}(a - \ell n) = \mathsf m_{S,w}(i).$$
Observing that $|\!\Ap(P_n; n)| = |\!\Ap(S; dn)| = n$ and that the elements of $\Ap(P_n; n)$ are all distinct modulo $n$ completes the proof.  
\end{proof}

\begin{remark}\label{r:aperysetrestriction}
The primary difficulty in generalizing Theorem~\ref{t:parametrizedapery} to the general setting considered in Sections~\ref{sec:linearfamilies} and~\ref{sec:minpres} is the ``non-surjectivity'' demonstrated in Example~\ref{e:uglymapping}.  The mapping utilized in the proof of Theorem~\ref{t:parametrizedapery} is indeed a specialization of the one established in the proof of Theorem~\ref{t:mesalemma}, but it specializes to a bijection in this case (i.e., when $w_1 = 1$).  
\end{remark}






Generalizations of \cite[Corollaries~4.2 and~4.3]{shiftedaperysets} follow immediately from Theorem~\ref{t:parametrizedapery}, and make use of the following observation from~\cite{shiftedaperysets}.  

\begin{prop}[{\cite[Proposition~3.4]{shiftedaperysets}}]\label{p:largeapery}
If $dn > F(S)$, then $\Ap(S;dn) = \{a_0, \ldots, a_{n-1}\}$, where
$$a_i = \left\{\begin{array}{ll}
di & \text{ if } di \in S; \\
di + dn & \text{ if } di \notin S.
\end{array}\right.$$
In particular, this holds whenever $n > W\!R^2$ as in Theorem~\ref{t:parametrizedapery}.  \end{prop}

\begin{cor}\label{c:frobquasi}
For $n > W\!R^2$, the function $n \to \mathsf F(P_n)$ has the form
$$\mathsf F(P_n) = \tfrac{w_k}{r_k}n^2 + a_1(n)n + a_0(n)$$
for some $r_k$-periodic functions $a_1(n)$ and $a_0(n)$.  
\end{cor}

\begin{proof}
Let $a$ denote the element of $\Ap(S;dn)$ for which $\mathsf m_{S,w}(-)$ is maximal.  Theorem~\ref{t:parametrizedapery} and Proposition~\ref{p:largeapery} imply
\begin{center}
$\begin{array}{r@{}c@{}l}
\mathsf F(P_n) &{}={}& \max(\Ap(P_n)) - n = a - n + \mathsf m_{S,w}(a) \cdot n,
\end{array}$
\end{center}
and Theorem~\ref{t:maxminquasi}\eqref{t:maxminquasi:minlen} implies $a + r_k$ is the element of $\Ap(S;dn + r_k)$ for which $\mathsf m_{S,w}(-)$ is maximal.  The quasilinearity of $\mathsf m_{S,w}(-)$ proves $n \mapsto \mathsf F(P_n)$ is quasiquadratic in $n$ with period $r_k$, and since the only degree-2 term in the above expression is $\mathsf m_{S,w}(a) \cdot n$, we obtain a leading coefficient identical to that of $\mathsf m_{S,w}(n)$, namely $w_k/r_k$.  
\end{proof}

\begin{cor}\label{c:genusquasi}
For $n > W\!R^2$, the function $n \mapsto \mathsf g(P_n)$ has the form
$$\mathsf g(P_n) = \tfrac{w_k}{2r_k}n^2 + b_1(n)n + b_0(n)$$
for some $r_k$-periodic functions $b_1(n)$ and $b_0(n)$.  
\end{cor}

\begin{proof}
By Remark~\ref{r:aperyfacts}, we can write
$$\mathsf g(P_n) = \sum_{a \in \Ap(P_n)} \left\lfloor\frac{a}{n}\right\rfloor.$$
Theorem~\ref{t:parametrizedapery} and Proposition~\ref{p:largeapery} then yield
$$\begin{array}{r@{}c@{}l}
\mathsf g(P_n)
&{}={}& \displaystyle \sum_{i \in \Ap(S;dn)} \left\lfloor\frac{i+ \mathsf m_{S,w}(i)n}{n}\right\rfloor
= \sum_{i \in \Ap(S;dn)} \left\lfloor\frac{i}{n}\right\rfloor 
+ \sum_{i \in \Ap(S;dn)}  \mathsf m_{S,w}(i)
\\[0.1em]
&{}={}& \displaystyle \sum_{t=1}^{n-1} \left\lfloor \frac{dt}{n} \right\rfloor
+ d \cdot \mathsf g(S) 
+ \sum_{\substack{i < n \\ di \in S}} \mathsf m_{S,w}(di)
+ \sum_{\substack{i \ge 0 \\ di \notin S}} \mathsf m_{S,w}(di+dn).
\end{array}$$
Each of the terms is eventually quasipolynomial in $n$. The first term is $d$-quasilinear in $n$, the second term is independent of $n$, and Theorem~\ref{t:parametrizedapery} guarantees that the last two terms are eventually $r_k$-quasiquadratic and $r_k$-quasilinear in $n$, respectively.  Since $d \mid r_k$, we conclude $n \mapsto \mathsf g(P_n)$ is quasiquadratic in $n$ with period $r_k$.  As for the leading term, the only degree-2 term in the above expression has successive $r_k$-differences
$$\sum_{\substack{i < n + r_k \\ di \in S}} \mathsf m_{S,w}(di) - \sum_{\substack{i < n \\ di \in S}} \mathsf m_{S,w}(di) = \sum_{j = 0}^{r_k-1} \mathsf m_{S,w}(dn + dj)$$
which are linear with leading coefficient $r_k(w_k/r_k) = w_k$.  This yields a leading coefficient of $w_k/2r_k$ for $n \mapsto \mathsf g(P_n)$, as claimed.  
\end{proof}

\begin{remark}\label{r:irreducibleasymptotic}
A numerical semigroup $S$ is called \emph{irreducible} if it is maximal with respect to containment among all numerical semigroups with Frobenius number $\mathsf F(S)$.  If $\mathsf F(S)$ is odd, this happens precisely when $\mathsf g(S) = (\mathsf F(S) + 1)/2$, and if $\mathsf F(S)$ is even, this happens precisely when $\mathsf g(S) = (\mathsf F(S) + 2)/2$.  Irreducible numerical semigroups have the smallest possible genus for their respective Frobenius number \cite[Chapter~3]{numerical}.  

As a consequence of the leading coefficients in Corollaries~\ref{c:frobquasi} and~\ref{c:genusquasi}, we obtain
$$\lim_{n \to \infty} \frac{\mathsf g(P_n)}{\mathsf F(P_n)} = \frac{1}{2},$$
which can be interpreted as saying $P_n$ is ``nearly'' irreducible for large $n$.  
\end{remark}

As a consequence, we obtain that for sufficiently large $n$, the numerical semigroup~$P_n$ satisfies Wilf's conjecture~\cite{wilfconjecture}, which is a longstanding open problem for numerical semigroups; see~\cite{wilfsurvey} for a survey of recent progress.  

\begin{cor}\label{c:wilfshifted}
For $n > W\!R^2$, the Wilf number of $P_n$, defined in \cite{delgadoconj} as
$$\mathsf W(P_n) = k(F(P_n) - g(P_n)) - (F(P_n) + 1),$$
is $r_k$-quasiquadratic in $n$.  In~particular, $\mathsf W(P_n)$ is positive for all suffiently large $n$, and thus $P_n$ satisfies Wilf's conjecture for each such $n$.  
\end{cor}

\begin{proof}
Apply Corollaries~\ref{c:frobquasi} and~\ref{c:genusquasi}.  
\end{proof}

Our final result concerns the (Cohen-Macaulay) type of $P_n$ for large $n$, which, just as in~\cite{shiftedaperysets}, we obtain from the pseudo-Frobenius numbers of $P_n$.  

\begin{defn}\label{d:pseudofrob}
An integer $m \ge 0$ is a \emph{pseudo-Frobenius number} of a numerical semigroup $T$ if $m \notin T$ but $m + n \in T$ for all positive $n \in T$.  Denote the set of pseudo-Frobenius numbers of $T$ by $\mathsf{PF}(T)$, and the \emph{type} of $T$ by $\mathsf t(T) = |\mathsf{PF}(T)|$.  
\end{defn}

\begin{cor}\label{c:pseudofrobshifted}
Given $n \in \ZZ_{\ge 0}$, let $F_n$ denote the set
$$F_n = \{i \in \Ap(S;dn) : a \equiv i \bmod n \text{ for some } a \in \mathsf{PF}(P_n)\}.$$
For $n > W\!R^2$, the map $F_n \to F_{n + r_k}$ given by
$$\begin{array}{rcl}
i
&\mapsto&
\left\{\begin{array}{ll}
i & \text{if } i \le dn \\
i + r_k & \text{if } i > dn
\end{array}\right.
\end{array}
$$
is a bijection.  In particular, there is a bijection $\mathsf{PF}(P_n) \to \mathsf{PF}(P_{n + r_k})$, meaning the function $n \mapsto t(P_n)$ is $r_k$-periodic for $n > W\!R^2$.  
\end{cor}

\begin{proof}
The proof is identical to that of \cite[Theorem~4.8]{shiftedaperysets}.  
\end{proof}



\section{Evidence of Conjecture~\ref{conj:main}}
\label{sec:evidence}

Now that we have seen the formal definition of a minimal presentation, we are ready to see an example of Conjecture~\ref{conj:main} for a more general (i.e., nonlinear) parametrized semigroup family.  Note that computational evidence for nonlinear families is harder to obtain since the substantially larger generators result in computations taking much longer to complete.  

\begin{example}\label{e:nonlinear}
Consider the parametrized family of semigroups 
$$P_n = \<m^2, m^2 + m + 1, m^2 + 2m + 1, m^2 + 2m + 3\>$$
and the following minimal presentations.  
\smaller
$$
\begin{array}{
@{}r@{\,\,\,}
l@{}r@{\,\,}r@{\,\,}r@{\,\,}r@{\,}r@{\,\,}r@{\,\,}r@{\,}r@{}l@{\,\,}
l@{}r@{\,\,}r@{\,\,}r@{\,\,}r@{\,}r@{\,\,}r@{\,\,}r@{\,}r@{}l@{\,\,}
l@{}r@{\,\,}r@{\,\,}r@{\,\,}r@{\,}r@{\,\,}r@{\,\,}r@{\,}r@{}l@{\,\,}
}
P_{52} : 
& (( & 0, & 0, & 27, & 0), & (0, & 1, & 0, & 26 & )),  
& (( & 0, & 3, & 26, & 0), & (2, & 0, & 0, & 27 & )), 
& (( & 0, & 4, & 0, & 0), & (2, & 0, & 1, & 1 & )), 
\\
& (( & 25, & 2, & 14, & 0), & (0, & 0, & 0, & 40 & )), 
& (( & 25, & 3, & 0, & 0), & (0, & 0, & 13, & 14 & )), 
& (( & 27, & 0, & 0, & 0), & (0, & 1, & 12, & 13 & ))
\\[0.5em]

P_{56} : 
& (( & 0, & 0, & 29, & 0), & (0, & 1, & 0, & 28 & )),  
& (( & 0, & 3, & 28, & 0), & (2, & 0, & 0, & 29 & )), 
& (( & 0, & 4, & 0, & 0), & (2, & 0, & 1, & 1 & )), 
\\
& (( & 27, & 2, & 15, & 0), & (0, & 0, & 0, & 43 & )), 
& (( & 27, & 3, & 0, & 0), & (0, & 0, & 14, & 15 & )), 
& (( & 29, & 0, & 0, & 0), & (0, & 1, & 13, & 14 & ))
\\[0.5em]

P_{60} : 
& (( & 0, & 0, & 31, & 0), & (0, & 1, & 0, & 30 & )),  
& (( & 0, & 3, & 30, & 0), & (2, & 0, & 0, & 31 & )), 
& (( & 0, & 4, & 0, & 0), & (2, & 0, & 1, & 1 & )), 
\\
& (( & 29, & 2, & 16, & 0), & (0, & 0, & 0, & 46 & )), 
& (( & 29, & 3, & 0, & 0), & (0, & 0, & 15, & 16 & )), 
& (( & 31, & 0, & 0, & 0), & (0, & 1, & 14, & 15 & ))
\\[0.5em]

P_{64} : 
& (( & 0, & 0, & 33, & 0), & (0, & 1, & 0, & 32 & )),  
& (( & 0, & 3, & 32, & 0), & (2, & 0, & 0, & 33 & )), 
& (( & 0, & 4, & 0, & 0), & (2, & 0, & 1, & 1 & )), 
\\
& (( & 31, & 2, & 17, & 0), & (0, & 0, & 0, & 49 & )), 
& (( & 31, & 3, & 0, & 0), & (0, & 0, & 16, & 17 & )), 
& (( & 33, & 0, & 0, & 0), & (0, & 1, & 15, & 16 & ))
\end{array}
$$
\normalsize
Unlike linear parametrized families, successive minimal presentations have more than just 2 coordinates consistently increasing, though the pattern in the relations is clear.  
\end{example}


\end{document}